\documentclass[12pt,twoside]{article}
\usepackage{amsmath,amsfonts,amssymb,amsthm,array,mathdots}
\usepackage{a4,enumitem,indentfirst}
\usepackage{etoolbox} 
\usepackage{stmaryrd} 

\usepackage{epsfig, color} 

\graphicspath{{figures/}}

\theoremstyle{plain}
\newtheorem{theo}{Theorem}
\newtheorem{prop}{Proposition}[section]

\newtheorem{coro}[prop]{Corollary}

\newtheorem{lemma}[prop]{Lemma}

\theoremstyle{definition}
\newtheorem{example}[prop]{Example}
\AtEndEnvironment{example}{\null\hfill$\diamond$}%

\theoremstyle{remark}
\newtheorem{rem}[prop]{Remark}
\AtEndEnvironment{rem}{\null\hfill$\diamond$}%

\newcommand\Flag{\operatorname{Flag}}
\newcommand\Frenet[1]{\mathfrak{F}_{#1}}

\newcommand{\Word}{{\mathbf W}}

\newcommand\SO{\operatorname{SO}}
\newcommand\SL{\operatorname{SL}}
\newcommand\GL{\operatorname{GL}}
\newcommand\Lo{\operatorname{Lo}}
\newcommand\Up{\operatorname{Up}}
\newcommand\slalgebra{\operatorname{\mathfrak{sl}}}
\newcommand\so{\operatorname{\mathfrak{so}}}
\newcommand\lo{\operatorname{\mathfrak{lo}}}
\newcommand\up{\operatorname{\mathfrak{up}}}
\newcommand\Spin{\operatorname{Spin}}
\newcommand\spin{\mathfrak{spin}}
\newcommand\Cliff{\operatorname{Cl}}
\newcommand\inv{\operatorname{inv}}
\newcommand\Inv{\operatorname{Inv}}
\newcommand\Diag{\operatorname{Diag}}
\newcommand\diag{\operatorname{diag}}
\newcommand\B{\operatorname{B}}
\newcommand\Quat{\operatorname{Quat}}
\newcommand\jacobi{\lambda}

\newcommand{\longacute}{\operatorname{acute}}
\newcommand{\longgrave}{\operatorname{grave}}

\newcommand{\longhat}{\operatorname{hat}}
\newcommand{\chop}{\operatorname{chop}}
\newcommand{\adv}{\operatorname{adv}}
\newcommand{\iti}{\operatorname{iti}}

\newcommand{\swminor}{\operatorname{swminor}}

\newcommand{\ba}{{\mathbf{a}}}

\newcommand{\nmesmo}{\llbracket  n \rrbracket}
\newcommand{\nmaisum}{\llbracket n+1 \rrbracket}

\newcommand{\transpose}{\top}

\newcommand{\mult}{\operatorname{mult}}

\newcommand{\NN}{{\mathbb{N}}}
\newcommand{\ZZ}{{\mathbb{Z}}}

\newcommand{\RR}{{\mathbb{R}}}

\newcommand{\Ss}{{\mathbb{S}}}
\newcommand{\BB}{{\mathbb{B}}}

\newcommand{\HH}{{\mathbb{H}}}

\newcommand{\cL}{{\cal L}}

\newcommand{\cU}{{\cal U}}

\newcommand{\cD}{{\cal D}}

\newcommand{\bi}{{\mathbf{i}}}
\newcommand{\bj}{{\mathbf{j}}}
\newcommand{\bL}{{\mathbf{L}}}
\newcommand{\bQ}{{\mathbf{Q}}}

\newcommand{\fa}{{\mathfrak a}}

\newcommand{\fh}{{\mathfrak h}}
\newcommand{\fu}{{\mathfrak u}}
\newcommand{\fl}{{\mathfrak l}}
\newcommand{\fn}{{\mathfrak n}}

\newcommand{\Pos}{\operatorname{Pos}}
\newcommand{\Neg}{\operatorname{Neg}}
\newcommand{\Bru}{\operatorname{Bru}}

\newcommand{\sign}{\operatorname{sign}}

\begin{document}

\title{Locally convex curves and \\
the Bruhat stratification of the spin group}
\author{Victor Goulart
\footnote{jose.g.nascimento@ufes.br;
Departamento de Matem\'atica, UFES,
Av. Fernando Ferrari 514; Campus de Goiabeiras, Vit\'oria, ES 29075-910, Brazil;
Departamento de Matem\'atica, PUC-Rio,  
R. Marqu\^es de S. Vicente 255, Rio de Janeiro, RJ 22451-900, Brazil. 
} 
\and Nicolau C. Saldanha\footnote{saldanha@puc-rio.br; Departamento de Matem\'atica, PUC-Rio.
}}
\date{\today}

\maketitle

\begin{abstract}

{
We study the lifting of the Schubert stratification of the homogeneous space of complete real flags of $\RR^{n+1}$ to its universal covering group $\Spin_{n+1}$. We call the lifted strata the \emph{Bruhat cells} of $\Spin_{n+1}$, in keeping with the homonymous classical decomposition of reductive algebraic groups. We present explicit  parameterizations for these Bruhat cells in terms of minimal-length expressions $\sigma=a_{i_1}\cdots a_{i_k}$ for permutations 
$\sigma\in S_{n+1}$ in terms of the $n$ generators $a_i=(i,i+1)$. These parameterizations are compatible with the 
Bruhat orders in the Coxeter-Weyl group $S_{n+1}$. 
This stratification is an important tool in the study of locally convex curves;
we present a few such applications. 
}

\medskip 

\end{abstract}

\section{Introduction}
\label{sect:intro}

Fix $n\in\NN$, $n\geq 2$, and consider the group $\SO_{n+1}$ of 
unit determinant real orthogonal matrices of order $n+1$ 
and its universal 
covering group $\Spin_{n+1}$. 
The latter can also be described in terms of the real Clifford algebra
induced by the standard Euclidean inner product 
of $\RR^{n+1}$ 
\cite{Atiyah-Bott-Shapiro, Lawson-Michelsohn}.
Familiarity with Clifford algebras is not required to read this paper though, 
since the required facts will be obtained from scratch. 

Let $S_{n+1}$ be the group of permutations of the set $\nmaisum=\{1,2,\ldots,n+1\}$. 
We denote the action of $S_{n+1}$ on $\nmaisum$ by 
$(\sigma,k)\mapsto k^\sigma$ (rather than $\sigma(k)$), so that 
$k^{\sigma_1\sigma_2}=(k^{\sigma_1})^{\sigma_2}$. 
We regard $S_{n+1}$ as the Coxeter-Weyl group $A_n$  
generated by the $n$ transpositions 
$a_1 = (1,2)$, $a_2 = (2,3)$, \dots, $a_n = (n,n+1)$. 

A \emph{reduced word} for a permutation $\sigma\in S_{n+1}$ is 
an expression of $\sigma$ as a product of the generators $a_j$ with 
minimal number of factors. 
This is $\inv(\sigma)=
\operatorname{card}
\left(\{(i,j)\in\nmaisum^2\,\vert\,(i<j)\land(i^\sigma>j^\sigma)\}\right)$, the number of inversions of $\sigma$.

Let $\B_{n+1}$ be the hyperoctahedral group of signed permutation matrices of order $n+1$, i.e., orthogonal matrices $P$ such that there exists a permutation $\sigma\in S_{n+1}$ with 
$e_j^\top P = \pm e^\top_{j^{\sigma}}$ for all $j\in\nmaisum$.
Here and henceforth, $(e_1,\ldots,e_{n+1})$ is the 
canonical basis of $\RR^{n+1}$. 
Let $\B_{n+1}^{+} = \B_{n+1} \cap \SO_{n+1}$ and 
$\Diag^{+}_{n+1} \subset \B^+_{n+1}$ 
be the normal subgroup of diagonal matrices, 
isomorphic to $\{ \pm 1 \}^n$. 
We have $\B^+_{n+1}/\Diag^+_{n+1}\approx S_{n+1}$, 
the quotient map being denoted by $P\mapsto\sigma_{P}$.
Lifting by the covering map $\Pi:\Spin_{n+1}\to\SO_{n+1}$, 
we have the exact sequences 
\[1\to\Quat_{n+1}\hookrightarrow\widetilde \B^+_{n+1}\stackrel{\sigma}{\to} S_{n+1}\to1,\quad 1\to\{\pm 1\}\hookrightarrow\Quat_{n+1}\stackrel{\Pi}{\to} \Diag^+_{n+1}\to1,\]
where we define $\widetilde \B^+_{n+1}=\Pi^{-1}[\B^+_{n+1}]$ and $\Quat_{n+1}=\Pi^{-1}[\Diag^+_{n+1}]$. 

Recall that $\Spin_{3}$ is isomorphic to $\Ss^{3}\subset\HH$, 
the group of quaternions with unit norm. 
Under this identification, we have 
$\operatorname{Quat}_{3} =
\{\pm 1,\pm\mathbf{i}, \pm \mathbf{j}, \pm\mathbf{k}\}$. 
The groups $\Quat_{n+1}$ are thus generalizations 
of the classical quaternion group $Q_8$, hence the notation;  
the group $\Quat_{n+1}$ is a subgroup of index $2$
of the \emph{Clifford group} as defined in \cite{Lawson-Michelsohn}
(notice that \cite{Atiyah-Bott-Shapiro} uses this term differently).

We now describe generators for $\widetilde\B^+_{n+1}$ 
and $\Quat_{n+1}$ closely related to the Coxeter generators $a_j$ of $S_{n+1}$. 
For each $j\in\nmesmo$, let
$\fa_j=e_{j+1}e_j^\top-e_je_{j+1}^\top\in\so_{n+1}$ be 
the matrix whose only nonzero entries are 
$(\fa_j)_{j+1,j}=+1$ and $(\fa_j)_{j,j+1}=-1$. 
Set $\alpha_j(\theta)=\exp(\theta\fa_j)$,
a 1-parameter subgroup of $\SO_{n+1}$.
We denote by the same symbol the lift to $\Spin_{n+1}$, 
so that $\alpha_j:\RR\to\Spin_{n+1}$ is also a 1-parameter subgroup. 
Set 
$\acute a_j=\alpha_j(\frac\pi2)\in\widetilde\B^+_{n+1}$ and 
$\hat a_j=\acute a_j^2\in\Quat_{n+1}$. 
Notice that $\acute a_j^4=\hat a_j^2=-1$  
and $\sigma_{\acute a_j}=a_j\in S_{n+1}$. 
The elements $\acute a_j$ and $\hat a_j$, $j\in\nmesmo$, generate $\widetilde\B^+_{n+1}$ and $\Quat_{n+1}$, respectively. 
The group $\Spin_{n+1}$ can be interpreted as a subset of the 
associative algebra with basis $1, \hat a_1, \hat a_2, \hat a_1 \hat a_2, \hat a_3, \ldots$ 
(with the product inherited from $\Quat_{n+1}$). 
With the identification $\hat a_j=e_{j+1}e_{j}$
(in the notation of  \cite{Atiyah-Bott-Shapiro, Lawson-Michelsohn}),
this is the subalgebra $\Cliff^0_{n+1}\subset\Cliff_{n+1}$ 
of even elements of the Clifford algebra. 
In this algebra, we have, for instance,
\begin{equation}
\label{equation:clifford}
\alpha_j(\theta) =
\cos(\theta/2) + \hat a_j \sin(\theta/2),
\qquad
\acute a_j =(1+\hat a_j)/\sqrt{2};
\end{equation}
this point of view will not be necessary but is sometimes helpful.

Now consider the homogeneous space $\Flag_{n+1}$ of the complete real flags 
\[\{0\}\subset\RR v_1\subset \RR v_1\oplus \RR v_2\subset\cdots\subset \RR v_1\oplus\cdots\oplus\RR v_{n+1} =\RR^{n+1}\] 
of $\RR^{n+1}$. 
This is a smooth manifold diffeomorphic to each one of the 
following spaces of left cosets:  
$\GL_{n+1}/\Up_{n+1}\approx  
\SO_{n+1}/\Diag_{n+1}^+\approx 
\Spin_{n+1}/\Quat_{n+1}$ 
(here, $\GL_{n+1}=\GL(n+1,\RR)$ and $\Up_{n+1}$ is the subgroup of upper triangular matrices).  
The group $\Spin_{n+1}$ is the $2^{n+1}$-fold universal covering of $\Flag_{n+1}$. 

Recall the classical decomposition of $\Flag_{n+1}$ into the Schubert cells 
$\mathcal{C}_{\sigma}$, indexed by permutations $\sigma\in S_{n+1}$
\cite{Bernstein-Gelfand-Gelfand, Chevalley, Demazure, Konstant, Verma}. 
These cells, particularly the intersection of translated cells, have been extensively studied \cite{Fomin-Zelevinsky1, Rietsch, Shapiro-Shapiro-Vainshtein1, Shapiro-Shapiro-Vainshtein2}; 
see also \cite{Fomin-Zelevinsky2}.
The \emph{unsigned Bruhat cell} $\Bru_\sigma\subset\Spin_{n+1}$ is the preimage 
under the projection of the Schubert cell $\mathcal C_\sigma\subset\Flag_{n+1}$. 
Equivalently, for $z\in\Spin_{n+1}$ and $\sigma\in S_{n+1}$, we have 
$z\in\Bru_\sigma$ if and only if there exist $U_1, U_2\in\Up_{n+1}$ 
and $P\in\B_{n+1}$ such that 
$\Pi(z)=U_1 P U_2$ and $\sigma_P=\sigma$.  
We have $\overline{\Bru_{\sigma}}\subseteq
\overline{\Bru_{\rho}}$ if and only if 
$\sigma\leq\rho$ in the (strong) Bruhat order 
\cite{Bjorner-Brenti, Humphreys, Verma}:  
given $\sigma_0,\sigma_1\in S_{n+1}$, we write $\sigma_0\leq\sigma_1$ 
if and only if there is a reduced word for $\sigma_0$ in terms of the 
Coxeter generators $a_j$ that is a subexpression 
of a reduced word for $\sigma_1$. 
The relation $\leq$ is a
directed graded partial order with rank function $\inv$, 
minimum $e$ (the identity) and maximum $\eta:j\mapsto n+2-j$, 
$\eta=a_1 a_2 a_1 a_3 a_2 a_1 \cdots a_n a_{n-1}\cdots a_2 a_1$  
(called the Coxeter element, usually denoted by $w_0$).

Each connected component of an unsigned Bruhat cell contains exactly one element of $z\in\widetilde\B^+_{n+1}$. 
Following \cite{Saldanha-Shapiro}, 
we call the connected component of $z\in\widetilde\B^+_{n+1}$ in 
$\Bru_{\sigma_z}$ a \emph{signed Bruhat cell}, denoted by $\Bru_z$,  
and call the cell decomposition 
\[
\Spin_{n+1}=\bigsqcup_{z\in\widetilde\B^+_{n+1}}\Bru_z
\]
the \emph{Bruhat stratification} of $\Spin_{n+1}$. 
We describe this stratification using the elementary $UPU$ Bruhat decomposition of invertible matrices. 
Also, familiarity with Schubert calculus is not assumed.

The group $\Quat_{n+1}$ acts freely and transitively on the collection of connected components of an unsigned Bruhat cell by left multiplication. 
Thus, the following result yields explicit parameterizations for all the 
signed Bruhat cells of $\Spin_{n+1}$.

\begin{theo}
\label{theo:Bruhat} 
Given reduced words $a_{i_1}\cdots a_{i_k}< a_{i_1}\cdots a_{i_k}a_j$  
for consecutive permutations in $S_{n+1}$ and signs
$\varepsilon_1,\ldots,\varepsilon_k, \varepsilon\in\{\pm1\}$, 
set $z_1=
(\acute a_{i_1})^{\varepsilon_1}\cdots (\acute a_{i_k})^{\varepsilon_k}$, 
$z_0=z_1(\acute a_j)^{\varepsilon}\in\widetilde\B^+_{n+1}$. 
Given $q\in\Quat_{n+1}$, the map 
$\Phi: \Bru_{qz_1} \times (0,\pi) \to \Bru_{qz_0}$, 
$\Phi(z,\theta) = z\alpha_j(\varepsilon\theta)$,  
is a diffeomorphism. 
\end{theo}

A similar result for the case $a_{i_1}\cdots a_{i_k}< a_j a_{i_1}\cdots a_{i_k}$ 
is also available.

\begin{coro}
\label{coro:Bruhat1}
In the conditions of the theorem, i.e., with 
$z_1=
(\acute a_{i_1})^{\varepsilon_1}\cdots (\acute a_{i_k})^{\varepsilon_k}$, 
$z_0=z_1(\acute a_j)^{\varepsilon}\in\widetilde\B^+_{n+1}$and $q\in\Quat_{n+1}$, 
we have the inclusion 
$\overline{\Bru_{qz_1}}\subset\overline{\Bru_{qz_0}}$.
\end{coro}

\begin{coro}
\label{coro:Bruhat2}
Given $q\in\Quat_{n+1}$, a reduced word 
$a_{i_1}\cdots a_{i_k}\in S_{n+1}$, and signs
$\varepsilon_1,\ldots,\varepsilon_k\in\{\pm1\}$, 
the map 
$\Psi:(0,\pi)^k \to 
\Bru_{q(\acute a_{i_1})^{\varepsilon_1}\cdots (\acute a_{i_k})^{\varepsilon_k}}$ 
given by  
$\Psi(\theta_1, \ldots, \theta_k)=
q \alpha_{i_1}(\varepsilon_1\theta_1) \cdots 
\alpha_{i_k}(\varepsilon_k\theta_k)$ 
is a diffeomorphism.
\end{coro}

The reader might want to compare the previous results with 
\cite{Berenstein-Fomin-Zelevinsky}, dealing with 
totally positive matrices, particularly in nilpotent triangular groups. 
Totally positive matrices were introduced 
independently in \cite{Gantmacher-Krein} and 
\cite{Schoenberg} and have since found widespread 
applications \cite{Ando, Brenti, Karlin, Lusztig3}. 
The concept of totally positive elements has been 
generalized to a reductive group $G$ and its flag 
manifold by G. Lusztig \cite{Lusztig1, Lusztig2, Lusztig3} 
and to Grassmannians by A. Postnikov \cite{Postnikov1, Postnikov2}. 
Our particular definition 
is analogous to that of \cite{Berenstein-Fomin-Zelevinsky}: 
this one is a 
good reference for facts mentioned 
without proof, particularly in Section \ref{sect:totallypositive}.

The second author was first led to consider similar stratifications 
while studying the homotopy type of certain spaces of parametric curves 
in the sphere $\Ss^n$ \cite{Saldanha3, Saldanha-Shapiro}. 
A map $\Gamma:J\to\Spin_{n+1}$, defined on an interval $J\subseteq\RR$, 
is called a \emph{locally convex curve} 
\cite{Alves-Saldanha, Goulart-Saldanha, Saldanha3, Saldanha-Shapiro}
if it is absolutely continuous (hence differentiable almost everywhere) and its 
logarithmic derivative has the form 
\[(\Gamma(t))^{-1}\Gamma'(t)=\sum_{j\in\nmesmo}\kappa_j(t)\fa_j\]
(wherever it is defined), where $\kappa_1,\ldots,\kappa_n:J\to(0,+\infty)$ are positive functions. 

Given a smooth locally convex curve $\Gamma$,  
the smooth curve $\gamma:J\to\RR^{n+1}$, 
$\gamma(t)=\Pi(\Gamma(t))e_1$, 
satisfies 
$\det(\gamma(t),\gamma'(t),\ldots,\gamma^{(n)}(t))>0$ 
for all $t\in J$.
A smooth parametric curve $\gamma:J\to\RR^{n+1}$ satisfying the inequality above is also called \emph{(positive) locally convex} or \emph{(positive) nondegenerate} \cite{Goulart-Saldanha, Khesin-Ovsienko, Khesin-Shapiro2, Little}. Such a curve $\gamma$ can be lifted to a locally convex curve 
$\Frenet{\gamma}$ in $\SO_{n+1}$ (and therefore in $\Spin_{n+1}$) 
by taking the orthogonal matrix $\Frenet{\gamma}(t)$ whose column-vectors are the result of applying the Gram-Schmidt algorithm to the ordered basis $(\gamma(t),\gamma'(t),\ldots,\gamma^{(n)}(t))$ of $\RR^{n+1}$. 
The orthogonal basis of $\RR^{n+1}$ thus obtained is the (generalized) Frenet frame of the space curve $\gamma$. The coefficients $\kappa_1,\ldots,\kappa_n$ of the logarithmic derivative of $\Frenet{\gamma}$ are the generalized curvatures of $\gamma$. 
The term locally convex comes from the fact that a nondegenerate curve 
$\gamma:J\to\RR^{n+1}$ can be partitioned into finitely many 
\emph{convex} arcs, i.e., arcs that intersect any $n$-dimensional subspace 
of $\RR^{n+1}$ at most $n$ times (with multiplicities taken into account).

A combinatorial approach to the topology of certain spaces of locally convex curves with fixed endpoints was put forward in the Ph.D. thesis \cite{Goulart, Goulart-Saldanha} of the first author, advised by the second. It relies strongly on the Bruhat stratification of $\Spin_{n+1}$ (particularly Theorem \ref{theo:Bruhat} above) and on several properties of the intersection of its translated cells with each other and with convex arcs. 
Some of these properties are proved in the present paper, e.g.,  
the next result, which gives a transversality condition
between smooth locally convex curves and Bruhat cells. 

\goodbreak

\begin{theo}
\label{theo:pathcoordinates}
Consider $z_0 \in \widetilde\B_{n+1}^{+} \subset \Spin_{n+1}$,
$\sigma = \sigma_{z_0} \in S_{n+1}$,
$\sigma \ne \eta$, $k = \inv(\eta) - \inv(\sigma) > 0$.
There exist an open neighborhood $\cU_{z_0}$ 
of the non-open signed Bruhat cell $\Bru_{z_0}$ in $\Spin_{n+1}$ 
and a smooth map 
$f = (f_1, \ldots, f_k): \cU_{z_0} \to \RR^k$
with the following properties.
For all $z \in \cU_{z_0}$,
$z \in \Bru_{z_0}$ if and only if
$f(z) = 0$. 
For all $z \in \cU_{z_0}$,
the derivative $Df(z)$ is surjective.
For any smooth locally convex curve
$\Gamma: (-\epsilon,\epsilon) \to \cU_{z_0}$
we have $(f_k \circ \Gamma)'(t) > 0$ for all $t \in (-\epsilon,\epsilon)$.
\end{theo}

\goodbreak

In other words, 
we introduce slice coordinates 
$(u_1,\ldots,u_{\inv(\sigma)},x_1,\ldots,x_k)$ 
in an open neighborhood $\cU_{z_0}$ of the non-open 
signed Bruhat cell $\Bru_{z_0}$, such that  
$\Bru_{z_0}=\{z\in\cU_{z_0}\,\vert\,x_1=\cdots=x_k=0\}$  
and the coordinate $x_k$ increases along every locally convex curve. 
This explicit construction is used in \cite{Goulart-Saldanha} to 
describe certain (infinite-dimensional) collared topological manifolds 
of locally convex curves crossing $\Bru_{z_0}$.

Set $\grave a_j = (\acute a_j)^{-1}$.
For a reduced word 
$\sigma=a_{i_1}\cdots a_{i_k}$,
set, as in Section \ref{sect:sign},
\begin{equation}
\label{equation:acutegrave}
\begin{gathered}
\acute \sigma= \longacute(\sigma) = \acute a_{i_1}\cdots \acute a_{i_k},
\quad
\grave \sigma= \longgrave(\sigma) =
\grave a_{i_1}\cdots \grave a_{i_k}\in\widetilde\B^+_{n+1}, \\
\hat \sigma=\longhat(\sigma) = \acute\sigma(\grave\sigma)^{-1}\in\Quat_{n+1}.
\end{gathered}
\end{equation}
The maps $\chop, \adv: \Spin_{n+1} \to \acute\eta \Quat_{n+1} \subset
\widetilde \B_{n+1}^{+}$
are defined by 
\begin{equation}
\label{equation:chopadvance}
\adv(z)=q_a \acute\eta, \quad
\chop(z) = q_c \grave\eta, \quad
z\in\Bru_{z_0}\subset\Bru_{\sigma_0}, \quad 
z_0=q_a \acute\sigma_0=q_c \grave\sigma_0,
\end{equation}
where, of course, $\sigma_0=\sigma_{z_0} \in S_{n+1}$ and $q_a,q_c\in\Quat_{n+1}$.  
For  $\rho_0=\eta\sigma_0$, we have
$\adv(z)=z_0 \longacute(\rho_0^{-1})=z_0 (\grave\rho_0)^{-1}$ 
and $\chop(z)\acute\rho_0=z_0$. 
In particular, $\adv(z)=\chop(z)\hat\rho_0$.

\begin{theo}
\label{theo:chopadvance}
For $z \in \Spin_{n+1}$,
let $\Gamma: (-\epsilon, \epsilon) \to \Spin_{n+1}$
be a locally convex curve such that $\Gamma(0) = z$.
There exists $\epsilon_a \in (0,\epsilon)$ such that
for all $t \in (0,\epsilon_a]$, $\Gamma(t) \in \Bru_{\adv(z)}$.
There exists $\epsilon_c \in (0,\epsilon)$ such that
for all $t \in [-\epsilon_c,0)$, $\Gamma(t) \in \Bru_{\chop(z)}$.
\end{theo}

The chopping map was introduced in \cite{Saldanha-Shapiro}, 
where a different combinatorial description is given,
with an emphasis on $\SO_{n+1}$.
Also, the topological claim of the theorem was 
proved for smooth locally convex curves. 
The notations $\ba = \grave\eta = \chop(1)$ 
and $A=\Pi(\ba)$ are used there; 
$A$ is called the \emph{Arnold matrix}.

Given a locally convex curve $\Gamma:J\to\Spin_{n+1}$, let 
$m_j(t)=m_{\Gamma;j}(t)$ be the determinant of the 
southwest $j\times j$ block of $\Pi(\Gamma(t))$, so that 
$m_j(t)$ is a minor of $\Pi(\Gamma(t))$. 
Given a permutation $\sigma\in S_{n+1}$ and $j\in\nmesmo$, 
we define the \emph{multiplicity}
$\mult_j(\sigma)=1^\sigma+\cdots+j^\sigma-(1+\cdots+j)$. 
Another important result is the following.

\begin{theo}
\label{theo:mult}
Let $\Gamma:J\to\Spin_{n+1}$ be a smooth locally convex curve. 
Consider $t_0 \in J$ and $\sigma\in S_{n+1}$. 
We have $\Gamma(t_0) \in \Bru_{\eta\sigma}$ if and only if, 
for all $j\in\nmesmo$, 
$t=t_0$ is a zero of $m_j(t)$ of multiplicity 
$\mult_j(\sigma)$.   
\end{theo}

In the statement above we adopt the convention that a ``zero" of multiplicity zero is no zero at all, i.e., it is a value $t=t_0$ in the domain of a function $f(t)$ such that $f(t_0)\neq 0$.

In Section \ref{sect:symmetric} we review some basics of the symmetric group. 
In Section \ref{sect:sign}, 
we study the group $\widetilde\B^{+}_{n+1}$. 
We are particularly interested in the maps 
\[\longacute, \longgrave:S_{n+1}\to\widetilde\B^{+}_{n+1}, \qquad\longhat:S_{n+1}\to\Quat_{n+1}\subset\widetilde\B^+_{n+1},\]
defined by Equation \ref{equation:acutegrave}.  
In Section \ref{sect:triangle} we introduce 
triangular systems of coordinates in large open subsets 
$\cU_{z_0}$ of the group $\Spin_{n+1}$ and study 
the so called \emph{convex curves} 
in the nilpotent lower triangular group $\Lo^{1}_{n+1}$.
In Section \ref{sect:totallypositive} we recall the concept 
of totally positive matrices. More generally, we define the subsets 
$\Pos_\sigma, \Neg_\sigma\subset\Lo^{1}_{n+1}$ 
for $\sigma\in S_{n+1}$. 
In Section \ref{sect:bruhatcell} we prove Theorems \ref{theo:Bruhat}, 
\ref{theo:pathcoordinates} and \ref{theo:chopadvance} 
and related results.  
Section \ref{sect:mult} contains the proof of Theorem \ref{theo:mult}. 
Section \ref{sect:finalremarkso} mentions applications of the results of 
the present paper in \cite{Goulart-Saldanha, Saldanha-Shapiro-Shapiro} 
and work in progress.

This paper contains follow-up material inspired
by the Ph. D. thesis of the first author,
advised by the second author
and co-advised by Boris Khesin, University of Toronto.
Both authors would like to thank:  
Em\'ilia Alves, 
Boris Khesin, 
Ricardo Leite, 
Carlos Gustavo Moreira, 
Paul Schweitzer, 
Boris Shapiro, 
Michael Shapiro, 
Carlos Tomei, 
David Torres, 
Cong Zhou 
and 
Pedro Z\"{u}lkhe 
for helpful conversations and 
the referee for a careful report.
We also thank
the University of Toronto and the University of Stockholm 
for the hospitality during visits.  

Both authors thank CAPES, CNPq and FAPERJ (Brazil) for financial support.
More specifically, the first author benefited from
CAPES-PDSE grant 99999.014505/2013-04 
during his Ph. D. and also 
CAPES-PNPD post-doc grant 88882.315311/2019-01.

\section{The symmetric group}
\label{sect:symmetric}


Two usual notations for a permutation $\sigma\in S_{n+1}$ are: 
as a product of Coxeter generators 
$a_1 = (1, 2),\ldots, a_i = (i, i+1), \ldots a_n = (n,n+1)$; 
as a list of values 
$[1^\sigma\,2^\sigma\cdots n^\sigma\,(n+1)^\sigma]$,   
the so called \emph{complete notation}. 
For $n\leq4$, we write $a=a_1$, $b=a_2$, $c=a_3$, $d=a_4$.
For instance, $ab=a_1a_2 = [312] \in S_3$. 


For $\sigma \in S_{n+1}$, let $P_{\sigma}$ be 
the permutation matrix defined by 
$e_k^\transpose P_{\sigma} = e_{k^\sigma}^\transpose$;
for instance, for $n = 2$ we have:
\[ P_a=P_{a_1} = \begin{pmatrix} 0 & 1 & 0 \\ 1 & 0 & 0 \\ 0 & 0 & 1 \end{pmatrix},
\quad
P_b=P_{a_2} = \begin{pmatrix} 1 & 0 & 0 \\ 0 & 0 & 1 \\ 0 & 1 & 0 \end{pmatrix}, \]
\[ P_{ab}=P_{a_1a_2} = P_{a_1} P_{a_2} =
\begin{pmatrix} 0 & 0 & 1 \\ 1 & 0 & 0 \\ 0 & 1 & 0 \end{pmatrix},
\quad
P_{ba}=P_{a_2a_1} = P_{a_2} P_{a_1} =
\begin{pmatrix} 0 & 1 & 0 \\ 0 & 0 & 1 \\ 1 & 0 & 0 \end{pmatrix}. \]
For $\sigma \in S_{n+1}$, let 
$\inv(\sigma) = |\Inv(\sigma)|$ be the number of inversions of $\sigma$;
the set of inversions is
$ \Inv(\sigma) = 
\{(i,j)\in\nmaisum^2\,\vert\,(i<j)\land(i^\sigma>j^\sigma)\}$. 
Recall that $\inv(\sigma)$ is also the length of a reduced word for 
$\sigma$ in terms of the generators $a_1,\ldots,a_n$. 
There exists a unique $\eta \in S_{n+1}$
with $\inv(\eta) = m = n(n+1)/2$, the Coxeter element
(a more common symbol for $\eta$ in the literature is $w_0$);
we have 
\[
\eta=a_1 a_2 a_1 a_3 a_2 a_1 \cdots a_n a_{n-1}\cdots a_2 a_1, \qquad  
P_{\eta} =
\begin{pmatrix} & & 1 \\ & \iddots & \\ 1 & & \end{pmatrix}. \]
A set $I\subseteq\{(i,j)\in \nmaisum^2\;\vert\;i<j\}$ 
is the set of inversions of a permutation $\sigma\in S_{n+1}$ 
if and only if for all $i,j,k\in\nmaisum$ with $i<j<k$,
the following two statements hold:
\begin{enumerate}
\item{if $(i,j),(j,k)\in I$ then $(i,k)\in I$;}
\item{if $(i,j),(j,k)\notin I$ then $(i,k)\notin I$.}
\end{enumerate}
Also, if $\rho = \sigma\eta$ then 
$\Inv(\sigma)\sqcup \Inv(\rho) = \Inv(\eta)$.

Let $\Up_{n+1}^{1}, \Lo_{n+1}^{1}$ 
be the nilpotent triangular groups 
of real upper and lower triangular matrices
with all diagonal entries equal to $1$.
For $\sigma \in S_{n+1}$, consider the subgroups
\begin{equation}
\label{equation:Upsigma}
\begin{aligned}
\Up_\sigma &= \Up_{n+1}^1 \cap (P_\sigma \Lo_{n+1}^{1} P_\sigma^{-1}) \\
& = \{ U \in \Up_{n+1}^1 \;|\; \forall i, j \in \nmaisum,
((i < j, U_{ij} \ne 0) \to ((i,j) \in \Inv(\sigma))\}, \\
\Lo_\sigma &= \Lo_{n+1}^1 \cap (P_\sigma \Up_{n+1}^{1} P_\sigma^{-1}) 
= (\Up_\sigma)^\transpose = P_\sigma \Up_{\sigma^{-1}} P_\sigma^{-1}, 
\end{aligned}
\end{equation}
affine subspaces of dimension $\inv(\sigma)$.
If $\rho = \sigma\eta$ then 
any $L \in \Lo_{n+1}^{1}$ can be written uniquely as
$L = L_1L_2$, $L_1 \in \Lo_\sigma$, $L_2 \in \Lo_\rho$.

As stated in the introduction, a \emph{reduced word} for $\sigma$ is an identity
\[ \sigma = a_{i_1}a_{i_2}\cdots a_{i_k}, \quad k = \inv(\sigma), \]
or, more formally, it is a finite sequence of indices
$(i_1,i_2,\ldots,i_k) \in \nmesmo^k$ 
satisfying the identity above.
Two reduced words for the same permutation $\sigma$
are connected by a finite sequence of local moves of two kinds:
\begin{gather}
\label{equation:reducedword1}
(\cdots,i,j,\cdots) \leftrightarrow (\cdots,j,i,\cdots), \quad
\quad |i-j|\ne 1; \\
\label{equation:reducedword2}
(\cdots,i,i+1,i,\cdots) \leftrightarrow (\cdots,i+1,i,i+1,\cdots); \quad
\end{gather}
corresponding to the identities $a_ia_j = a_ja_i$ for $|i-j| \ne 1$
and $a_i a_{i+1} a_i =  a_{i+1} a_i a_{i+1}$, respectively
(see \cite{Elnitsky, Humphreys}).

The (strong) Bruhat order $<$ defined in the introduction 
can also be defined as the transitive closure of 
a relation $\vartriangleleft$ defined in $S_{n+1}$ 
as follows: 
write $\sigma_0 \vartriangleleft \sigma_1$
if $\inv(\sigma_1) = \inv(\sigma_0) + 1$ and
$\sigma_1 = \sigma_0 (j_0j_1) = (i_0i_1) \sigma_0$;
here $i_0 < i_1$, $j_0 < j_1$,
$i_0^{\sigma_0} = j_0$, $i_1^{\sigma_0} = j_1$,
$i_0^{\sigma_1} = j_1$, $i_1^{\sigma_1} = j_0$.
We have $\sigma_0 \vartriangleleft \sigma_1$ if and only if 
$\sigma_0$ is an immediate predecessor of $\sigma_1$ 
in the Bruhat order.
We have $\sigma_0 < \sigma_k$ 
(with $k = \inv(\sigma_k) - \inv(\sigma_0)$)
if and only if there exist $\sigma_1, \ldots, \sigma_{k-1}$ with
$\sigma_0 \vartriangleleft \sigma_1 \vartriangleleft \cdots
\vartriangleleft \sigma_{k-1} \vartriangleleft \sigma_k.$

If $\sigma_1$ is written as $[1^{\sigma_1}\cdots(n+1)^{\sigma_1}]$,
it is easy to find its immediate predecessors:
look for integers $j_1 > j_0$ appearing in the list 
$[1^{\sigma_1}\cdots(n+1)^{\sigma_1}]$, 
$j_1$ to the left of $j_0$,
such that the integers which appear in the list between $j_1$ and $j_0$
are either larger than $j_1$ or smaller than $j_0$;
the permutation $\sigma_0 \vartriangleleft \sigma_1$ is then obtained
by switching the entries $j_1$ and $j_0$.
In the matrix $P_{\sigma_1}$, we must look for positive entries
$(i_0,j_1)$, $(i_1,j_0)$ such that the interior of the rectangle
with these vertices includes no positive entry. Then 
$P_{\sigma_0}$ is obtained by flipping these entries 
to the other corners of the rectangle while leaving the complement 
of the rectangle unchanged.

The strong Bruhat order must not be confused with the left and right
weak Bruhat orders.
The weak left Bruhat order $<_L$ 
is the transitive closure of the relation $\vartriangleleft_L$ 
defined as follows:
$\sigma_1 \vartriangleleft_{L} \sigma_0$ if $\sigma_1 \vartriangleleft \sigma_0$
and $\sigma_0 = a_i \sigma_1$ (for some $i$).
Equivalently, $\sigma_1 \le_L \sigma_0$ if
$\Inv(\sigma_1^{-1}) \subseteq \Inv(\sigma_0^{-1})$.
Similarly, 
$\sigma_1 \vartriangleleft_{R} \sigma_0$ if $\sigma_1 \vartriangleleft \sigma_0$
and $\sigma_0 = \sigma_1 a_j$ (for some $j$);
the transitive closure $\sigma_1 \le_R \sigma_0$
is characterized by $\Inv(\sigma_1) \subseteq \Inv(\sigma_0)$.
Notice that either $\sigma_1 \vartriangleleft_L \sigma_0$
or $\sigma_1 \vartriangleleft_R \sigma_0$
imply  $\sigma_1 \vartriangleleft \sigma_0$;
on the other hand,
$\sigma_1 = [2143] = a_1a_3 \vartriangleleft \sigma_0 = [4123] = a_1a_2a_3$,
but $\sigma_1 \not\le_L \sigma_0$ and $\sigma_1 \not\le_R \sigma_0$.
For more on Coxeter groups and Bruhat orders, see \cite{Bjorner-Brenti, Humphreys}.


\begin{lemma}
\label{lemma:aij}
Consider $\sigma \in S_{n+1}$ and $i, j \in \nmesmo$
such that $|i-j| > 1$.
Then $\sigma \vartriangleleft \sigma a_i$
if and only if
$\sigma a_j \vartriangleleft \sigma a_j a_i = \sigma a_i a_j$.
\end{lemma}

\begin{proof}
The condition 
$\sigma \vartriangleleft \sigma a_i$
is equivalent to $i^\sigma < (i+1)^\sigma$.
But $i^\sigma = i^{(\sigma a_j)}$
and $(i+1)^\sigma = (i+1)^{(\sigma a_j)}$,
proving the desired equivalence.
\end{proof}

Define
\[ \sigma_0 \vee e = \sigma_0, \qquad \sigma_0 \vee a_i = \begin{cases}
\sigma_0, & \text{if } \sigma_0 a_i \vartriangleleft \sigma_0; \\
\sigma_0 a_i, & \text{if } \sigma_0 \vartriangleleft \sigma_0 a_i. \end{cases} \]
A simple computation verifies that
\begin{gather*}
|i-j| \ne 1 \quad\implies\quad
(\sigma_0 \vee a_i) \vee a_j = (\sigma_0 \vee a_j) \vee a_i; \\
((\sigma_0 \vee a_i) \vee a_{i+1}) \vee a_i =
((\sigma_0 \vee a_{i+1}) \vee a_i) \vee a_{i+1}. 
\end{gather*}
We may therefore recursively define
\[ \sigma_1 \vartriangleleft \sigma_1 a_i \quad\implies\quad
\sigma_0 \vee (\sigma_1 a_i) = (\sigma_0 \vee \sigma_1) \vee a_i; \]
the previous remarks, together with the connectivity of reduced words
under the moves in Equations \ref{equation:reducedword1} and \ref{equation:reducedword2},
show that this is well defined.
Equivalently, $\sigma_0 \vee \sigma_1$ is the smallest $\sigma$
(in the strong Bruhat order) satisfying both
$\sigma_0 \le_R \sigma$ and $\sigma_1 \le_L \sigma$.
Notice that $S_{n+1}$ is not a lattice with the strong Bruhat order;
the $\vee$ operation above uses more than one partial order.
In general, we may have $\sigma_0 \vee \sigma_1 \ne \sigma_1 \vee \sigma_0$
and $\sigma_0 \vee \sigma_0 \ne \sigma_0$.
We do have associativity:
$(\sigma_0 \vee \sigma_1) \vee \sigma_2 =
\sigma_0 \vee (\sigma_1 \vee \sigma_2)$.

\begin{example}
\label{example:vee}
Take $n = 3$, $\sigma_0 = [2413]$,
$\sigma_1 = [2431] = cbca$.
We then have $\sigma_0\vee\sigma_1 =
(((\sigma_0 \vee c) \vee b) \vee c) \vee a =
((\sigma_0 \vee b) \vee c) \vee a =
(\sigma_0 b \vee c) \vee a =
\sigma_0 bc \vee a = \sigma_0 bca = \eta$.
\end{example}

Another useful representation of a permutation is in terms of its
\emph{multiplicities}, which we now define.
For $\sigma \in S_{n+1}$ and $k\in\nmesmo$, let 
\[ \mult_k(\sigma) = \sum_{j\in\llbracket k\rrbracket} (j^\sigma - j), \quad
\mult(\sigma) = \left( \mult_1(\sigma), \mult_2(\sigma), \ldots,
\mult_n(\sigma) \right). \]
With the convention $\mult_0(\sigma) = \mult_{n+1}(\sigma) = 0$,
we have $k^\sigma = k + \mult_k(\sigma) - \mult_{k-1}(\sigma)$,
so that the \emph{multiplicity vector} $\mult(\sigma)$ easily 
determines $\sigma$. The reason for calling $\mult_k(\sigma)$ 
a multiplicity is clear from Theorem \ref{theo:mult}.

If $d, \tilde d \in \NN^n$ we write $d \le \tilde d$
if, for all $k$, $d_k \le \tilde d_k$.
If $\sigma_0 \le \sigma_1$ (in the Bruhat order)
then $\mult(\sigma_0) \le \mult(\sigma_1)$
and $\inv(\sigma_0) \le \inv(\sigma_1)$.

\begin{example}
For $n = 5$, let $\sigma_0 = [432156]$ and $\sigma_1 = [612345]$.
We have $\mult(\sigma_0) = (3,4,3,0,0) \le \mult(\sigma_1) = (5,4,3,2,1)$
but $\inv(\sigma_0) = 6 > \inv(\sigma_1) = 5$.
For $n = 6$, let $\sigma_2 = [4321567]$ and $\sigma_3 = [7123456]$.
We have $\inv(\sigma_2) = \inv(\sigma_3) = 6$ 
and $\mult(\sigma_2) = (3,4,3,0,0,0) < \mult(\sigma_3) = (6,5,4,3,2,1)$.
\end{example}

\begin{lemma}
\label{lemma:lessdot}
Let $\sigma_0 \vartriangleleft \sigma_1$
with $\sigma_1 = (i_0i_1) \sigma_0 = \sigma_0 (j_0j_1)$.
Then
\[ \mult_k(\sigma_1) = 
\mult_k(\sigma_0) + (j_1 - j_0)\; [i_0 \le k < i_1]. \]
\end{lemma}

Here we use Iverson notation (or Iverson bracket):
if $\phi$ is a statement, then $[\phi] = 1$ if $\phi$ is true
and $[\phi] = 0$ if $\phi$ is false.
Thus, for instance,
\begin{equation}
\label{equation:iverson}
[i_0 \le k < i_1] =
\begin{cases} 1, & i_0 \le k < i_1, \\ 0, & \textrm{otherwise.} \end{cases}
\end{equation}

\begin{proof}
This is an easy computation.
\end{proof}

Let $\inv_i(\sigma) = |\Inv_i(\sigma)|$ where
\begin{equation}
\label{equation:invi}
\Inv_i(\sigma) = \{ j \;|\; i < j, i^\sigma > j^\sigma \}
= \{ j \;|\; (i,j) \in \Inv(\sigma) \};
\end{equation}
notice that $\Inv(\sigma) = \bigsqcup_i (\{i\} \times \Inv_i(\sigma))$
and therefore $\inv(\sigma) = \sum_i \inv_i(\sigma)$.

\begin{lemma}
\label{lemma:invpi}
For any $\sigma \in S_{n+1}$ and for any $i \in \nmaisum$
we have 
\[\inv_i(\sigma) - \inv_{i^\sigma}(\sigma^{-1})= i^\sigma - i =\mult_{i}(\sigma)-\mult_{i-1}(\sigma).\]
\end{lemma}

\begin{proof}
The permutation $\sigma$ restricts to a bijection between the two sets:
\begin{align*}
\{i+1, \ldots, n+1\} \smallsetminus \Inv_i(\sigma) 
&= \{ j \;|\; i < j, i^\sigma < j^\sigma \}, \\
\{i^\sigma+1, \ldots, n+1\} \smallsetminus \Inv_{i^\sigma}(\sigma^{-1}) 
&= \{ j' \;|\; i^{\sigma} < j', i < (j')^{\sigma^{-1}} \},
\end{align*}
with cardinalities
$n+1-i-\inv_i(\sigma)$ and $n+1-i^\sigma-\inv_{i^\sigma}(\sigma^{-1})$.
\end{proof}

The notion of multiplicity is closely related 
to a beautiful 1-1 correspondence, discovered 
by S. Elnitsky \cite{Elnitsky}, between commutation 
classes of reduced words for a permutation 
$\sigma\in S_{n+1}$ and the rhombic tilings of 
a certain (possibly degenerate) $2(n+1)$-gon 
associated to $\sigma$. 
This correspondence is an expedient way to obtain 
reduced words from complete notation.
An equivalent (if somewhat deformed) version 
of this construction is obtained by considering 
tesselations by parallelograms of the plane region 
$\mathcal{P}_{\sigma}$ between the graphs 
of $k\mapsto (2\mult_{k}(\sigma)-\mult_{k}(\eta))$ 
and $k\mapsto (-\mult_{k}(\eta))$. 
Under this deformation, the initial regular $2(n+1)$-gon is taken into the region $\mathcal{P}_{\eta}$ between the graphs of $\mult_\eta$ and $-\mult_{\eta}$. Given a decomposition of $\mathcal{P}_{\sigma_{0}}$ into $\inv(\sigma_{0})$ parallelograms, each one of them has a diagonal lying on one of the vertical lines $k=1,2,\cdots,n$. One then looks for an exposed, non imbricate piece to withdraw from the uppermost layer (there can be many of them to choose from). Suppose you pick a parallelogram $\mathcal{Q}_{1}$ crossed by the vertical line $k=j_{1}$. The plane region $\overline{\mathcal{P}_{\sigma_{0}}\smallsetminus\mathcal{Q}_{1}}$ is the $2(n+1)$-gon $\mathcal{P}_{\sigma_{1}}$ associated to the permutation $\sigma_{1}\vartriangleleft\sigma_{0}$ given by $\sigma_{0}=a_{i_{1}}\sigma_{1}$. Proceeding likewise with $\mathcal{P}_{\sigma_{1}}$ and so on, after $\inv(\sigma_{0})$ steps we arrive at a reduced word $\sigma_{0}=a_{i_{1}}a_{i_{2}}\cdots a_{i_{\inv(\pi)}}$.
An analogous procedure can be performed directly on the graph of $\mult_{\pi_{0}}$, as illustrated in Figure \ref{figure:multiplicity}.

\begin{figure}[ht]
\centering
\includegraphics[width=0.3\textwidth]{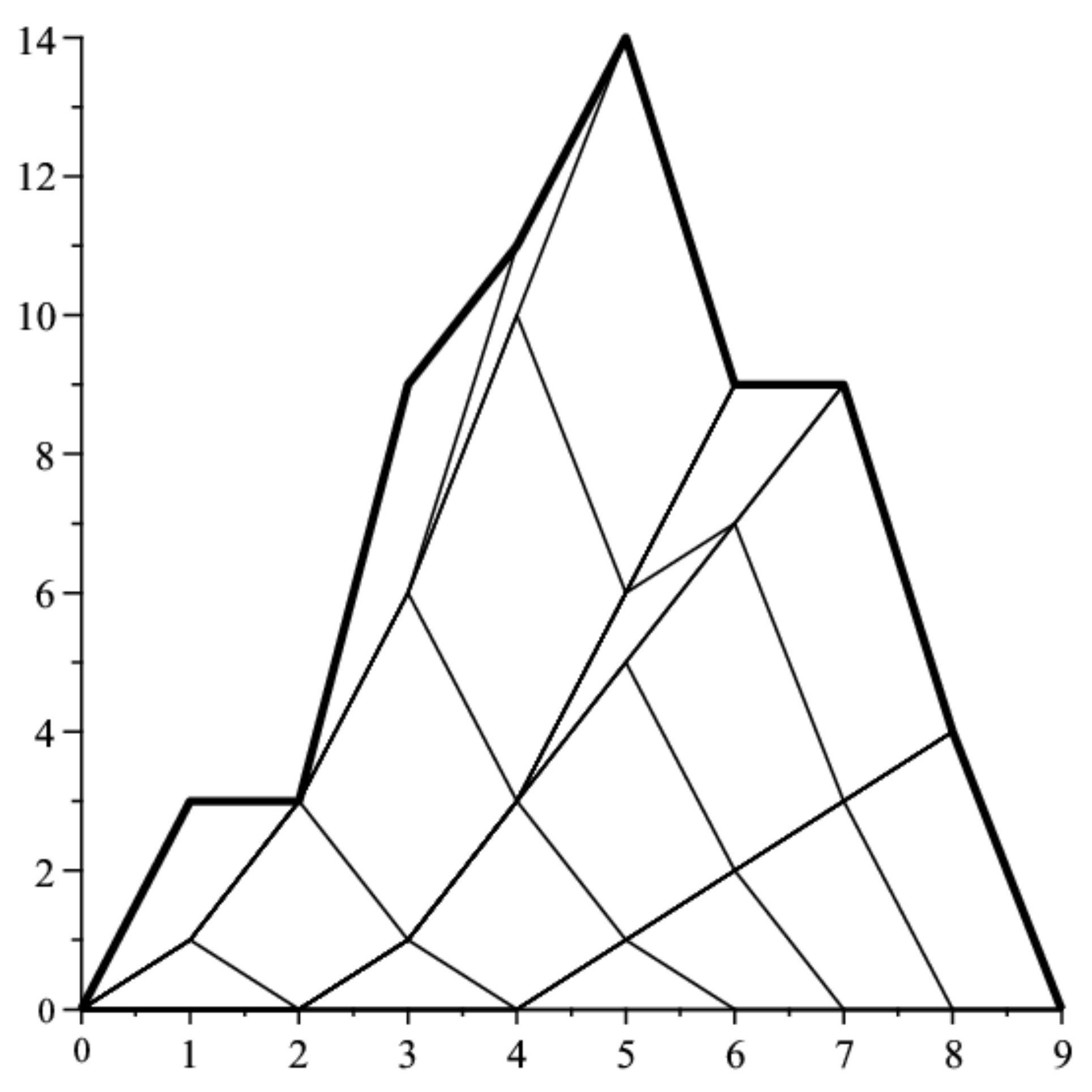}
\includegraphics[width=0.3\textwidth]{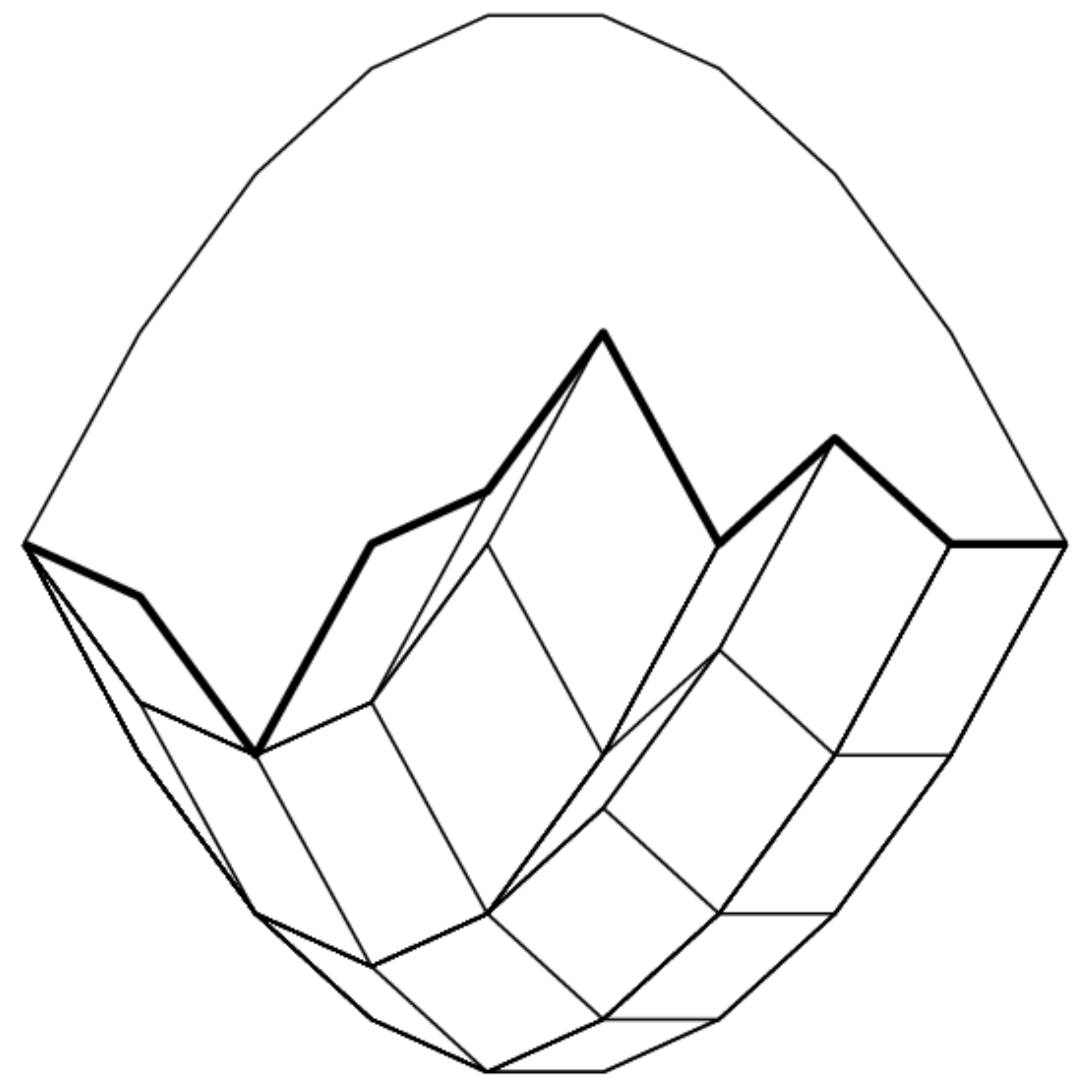}
\caption{Tilings of the graph of 
$k\mapsto \mult_{k}{\sigma}$ and of the 
Elnitsky's polygon $\mathcal{P}_{\sigma}$ for 
$\sigma=[429681735]\in S_{9}$ corresponding 
to the commutation class of the reduced word 
$\sigma=a_{1}a_{3}a_{4}a_{5}a_{4}a_{3}a_{2}a_{1}a_{6}a_{5}a_{7}a_{6}a_{5}a_{4}a_{3}a_{8}a_{7}a_{6}a_{5}$.}
\label{figure:multiplicity}
\end{figure}

\section{Signed permutations}
\label{sect:sign}

In this section we study the lift $\widetilde\B^+_{n+1}\subset\Spin_{n+1}$ of the index two subgroup $\B^+_{n+1}=\B_{n+1}\cap\SO_{n+1}$ of 
the hyperoctahedral group $\B_{n+1}$. 
Recall the surjective group homomorphism 
$\sigma:\widetilde\B^+_{n+1}\to S_{n+1}$, $z\mapsto\sigma_z$, and its 
kernel $\Quat_{n+1}$. 

The group $\B_{n+1}$ is a Coxeter group (whence the notation) 
with generators $P_{a_1},\ldots,P_{a_n}, R$, where 
$e_k^\top R=(-1)^{[k=1]}e_k^\top$, but we do not use this presentation
(the bracket $[k=1]$ is another example of Iverson bracket,
already seen in Equation \ref{equation:iverson}). 
Rather, consider the elements  
$\acute a_1,\ldots, \acute a_n\in\widetilde\B^+_{n+1}$ 
defined in the introduction by 
$\acute a_j=\exp\left(\frac\pi2\fa_j\right)$, 
$\fa_j=e_{j+1}e_j^\top-e_je_{j+1}^\top\in\so_{n+1}\approx\spin_{n+1}$. 
Also, recall the elements $\hat a_j=\acute a_j^2\in\Quat_{n+1}$.

\begin{lemma}
\label{lemma:stepacute}
The following identities hold:
\[ |i-j| \ne 1 \quad \implies \quad
\acute a_j\acute a_i = \acute a_i\acute a_j, \;
\hat a_j \acute a_i = \acute a_i \hat a_j, \;
\hat a_j \hat a_i = \hat a_i \hat a_j; \]
\[ \acute a_i\acute a_{i+1}\acute a_i = \acute a_{i+1}\acute a_i\acute a_{i+1};
\qquad
(\acute a_i)^{-1}\acute a_{i+1}(\acute a_i)^{-1}
= \acute a_{i+1}(\acute a_i)^{-1}\acute a_{i+1}; \]
\[ |i-j| = 1 \quad \implies \quad
\hat a_j \acute a_i =  (\acute a_i)^{-1} \hat a_j, \;
\hat a_j \hat a_i = - \hat a_i \hat a_j. \]
\end{lemma}

\begin{proof}
These are simple computations with any point of view;
they are particularly easy using the Clifford algebra $\Cliff_{n+1}^0$,
as discussed in the introduction near Equation \ref{equation:clifford}.
\end{proof}

Each element $q \in \Quat_{n+1}$ can be written
uniquely as
\[q = \pm \hat a_1^{\varepsilon_1} \hat a_2^{\varepsilon_2} \cdots
\hat a_n^{\varepsilon_n}, \quad
\varepsilon_i \in \{0,1\}.\] 
In particular, the elements $\hat a_1,\ldots,\hat a_n$ generate $\Quat_{n+1}$.
Furthermore, if $z \in \widetilde \B_{n+1}^{+}$ and
$\sigma_{z}= a_{i_1} \cdots a_{i_k} \in S_{n+1}$,
take $z_1 = \acute a_{i_1} \cdots \acute a_{i_k} \in \widetilde \B_{n+1}^{+}$:
we have $\sigma_{z} = \sigma_{z_1}$ and therefore
$z = q z_1$ with $q \in \Quat_{n+1}$.
In particular, the elements $\acute a_1,\ldots,\acute a_n$ 
generate $\widetilde \B_{n+1}^{+}$.
We make this construction more systematic.

\begin{lemma}
\label{lemma:goodacute}
If $\sigma \in S_{n+1}$ is expressed by two reduced words
$\sigma = a_{i_1} \cdots a_{i_k} = a_{j_1} \cdots a_{j_k}$
then
$\acute a_{i_1} \cdots \acute a_{i_k} = \acute a_{j_1} \cdots \acute a_{j_k}$.
\end{lemma}

\begin{proof}
Both moves (as in Equations \ref{equation:reducedword1} and \ref{equation:reducedword2})
are taken care of by Lemma \ref{lemma:stepacute}.
\end{proof}

Let $\grave a_i = (\acute a_i)^{-1}$.
For $\sigma \in S_{n+1}$, take a reduced word
$\sigma = a_{i_1} \cdots a_{i_k}$ and set 
\[ \longacute(\sigma) = \acute\sigma = \acute a_{i_1} \cdots \acute a_{i_k};
\qquad
\longgrave(\sigma) = \grave\sigma = \grave a_{i_1} \cdots \grave a_{i_k}, \]
as in Equation \ref{equation:acutegrave}.
Lemma \ref{lemma:goodacute} shows that the maps 
$\longacute,\longgrave: S_{n+1} \to \widetilde \B_{n+1}^{+}$ are well defined.
Notice that these maps are not homomorphisms.
Similarly, non-reduced words do not work
in the above formulas for $\acute\sigma$ and $\grave\sigma$. 
Also, define 
\[ \longhat(\sigma) = \hat\sigma =
\acute\sigma (\grave\sigma)^{-1} =
\acute a_{i_1} \cdots \acute a_{i_k} \acute a_{i_k} \cdots \acute a_{i_1}, \]
so that $\hat\sigma \in \Quat_{n+1}$ for all $\sigma \in S_{n+1}$.
Notice that these notations are consistent
with the previously introduced special cases $\acute a_i$ and $\hat a_i$.

\begin{lemma}
\label{lemma:pihat}
Consider $\sigma \in S_{n+1}$ and set
$\acute{P}=\Pi(\acute{\sigma})\in \B^+_{n+1}$. We have
\[e^{\transpose}_{i}\acute{P}=
(-1)^{\inv_{i}(\sigma)}e^{\transpose}_{i^{\sigma}}\,,
\qquad \acute{P}e_{j}=
(-1)^{\inv_{j^{\sigma^{-1}}}(\sigma)}e_{j^{\sigma^{-1}}}\]
and therefore $\acute{P}_{ij}=
e^{\transpose}_{i}\acute{P}e_{j}=
(-1)^{\inv_{i}(\sigma)}[j=i^{\sigma}].$ 
The nonzero entries of $\hat P = \Pi(\hat\sigma) \in \Diag_{n+1}^{+}$ are
\[ (\hat P)_{ii} = (-1)^{\inv_{i}(\sigma) + \inv_{i^\sigma}(\sigma^{-1})}=
(-1)^{i + i^\sigma}
= (-1)^{\mult_{i-1}(\sigma)+\mult_{i}(\sigma)}. \]
We also have $ \hat\sigma = \pm \hat a_1^{\mult_1(\sigma)}\cdots
\hat a_n^{\mult_n(\sigma)}$.
\end{lemma}

The expression $[j=i^{\sigma}]$ in the statement above is 
another use of Iverson bracket. 
Recall that $\inv_i(\sigma) = |\Inv_i(\sigma)|$,
where $\Inv_i(\sigma)$ 
is defined in Equation \ref{equation:invi}.



\begin{proof}
The first expression for the diagonal entries of 
$\hat{P}= \Pi(\acute\sigma)\Pi(\longacute(\sigma^{-1}))$ 
follows directly from the first two formulae, 
which we now prove by induction on $\inv(\sigma)$. 
The base cases $\inv(\sigma) \le 1$ are easy. 
Assume $\sigma_1 \vartriangleleft \sigma = a_k \sigma_1$, 
so that $\acute{P}=\Pi(\acute{a}_{k})\acute{P}_{1}$, 
where $\acute{P}_{1}=\Pi(\acute{\sigma}_{1})$. 
By the induction hypotheses we have 
\[e^{\transpose}_{i}\acute{P}=
e^{\transpose}_{i}\Pi(\acute{a}_{k})\acute{P}_{1}= 
(-1)^{\inv_{i}(a_{k})}e^{\transpose}_{i^{a_{k}}}\acute{P}_{1} =
(-1)^{\inv_{i}(a_{k})+\inv_{i^{a_{k}}}(\sigma_{1})}e^{\transpose}_{i^{\sigma}}.\]
To see that $\inv_{i}(\sigma)=\inv_{i}(a_{k})+\inv_{i^{a_{k}}}(\sigma_1)$ for all values of $i\in\nmaisum$, consider separately the cases $i<k$, $i=k$, $i=k+1$ and $i>k+1$. The second formula is similar. The alternate expressions for $\hat{P}_{ii}$ are obtained via Lemma \ref{lemma:invpi}.
The last of these expressions imply that 
$\hat P=\Pi(\hat\sigma)=
\Pi(\hat a_1^{\mult_1(\sigma)}\cdots\hat a_n^{\mult_n(\sigma)})$, 
and therefore, 
$\hat\sigma=
\pm\hat a_1^{\mult_1(\sigma)}\cdots\hat a_n^{\mult_n(\sigma)}$.
\end{proof}

If $\sigma_1 \vartriangleleft \sigma_0 = a_i \sigma_1$ 
then, by definition,
\[ \acute\sigma_0 = \acute a_i \acute\sigma_1, \qquad
\hat\sigma_0 = \acute a_i \hat\sigma_1 \acute a_i. \]
We show how to obtain a different recursive formula for $\hat\sigma_0$.

\begin{lemma}
\label{lemma:parityquat}
Let $q \in \Quat_{n+1}$ and $E = \Pi(q) \in \Diag^+_{n+1}$; write
\[ q = \pm \hat a_1^{\varepsilon_1} \cdots \hat a_n^{\varepsilon_n}, \quad \varepsilon_1,\ldots,\varepsilon_n \in \ZZ. \]
With the convention $\varepsilon_0 = \varepsilon_{n+1} = 0$, we have:
\begin{enumerate}
\item{If $\varepsilon_{i-1} + \varepsilon_{i+1}$ is odd then
$q \acute a_i = \grave a_i q$,
$q \hat a_i = - \hat a_i q$, $E_{i+1,i+1} = - E_{i,i}$. }
\item{If $\varepsilon_{i-1} + \varepsilon_{i+1}$ is even then
$q \acute a_i = \acute a_i q$,
$q \hat a_i =  \hat a_i q$, $E_{i+1,i+1} = E_{i,i}$. }
\end{enumerate}
\end{lemma}

\begin{proof}
From Lemma \ref{lemma:stepacute}, 
\[ \hat a_j \acute a_i = \begin{cases}
\acute a_i \hat a_j, & |i-j| \ne 1, \\
\grave a_i \hat a_j, & |i-j| = 1; \end{cases} \qquad
\hat a_j \grave a_i = \begin{cases}
\grave a_i \hat a_j, & |i-j| \ne 1, \\
\acute a_i \hat a_j, & |i-j| = 1; \end{cases} \]
these imply the formulas for $q \acute a_i$.
We then have, for $\varepsilon_{i-1} + \varepsilon_{i+1}$ even,
\[ q \hat a_i = q \acute a_i \acute a_i =  \acute a_i q \acute a_i =
\acute a_i \acute a_i q = \hat a_i q \]
and, for $\varepsilon_{i-1} + \varepsilon_{i+1}$ odd,
\[ q \hat a_i = q \acute a_i \acute a_i =  \grave a_i q \acute a_i =
\grave a_i \grave a_i q = -\hat a_i q. \]
Finally, notice that $\varepsilon_{i-1} + \varepsilon_{i+1}$ even implies
$E \Pi(\acute a_i) = \Pi(\acute a_i) E$ and therefore
$E_{i+1,i+1} = E_{i,i}$;
conversely, $\varepsilon_{i-1} + \varepsilon_{i+1}$ odd implies
$E \Pi(\acute a_i) =\Pi((\acute a_i)^{-1}) E = (\Pi(\acute a_i))^{-1} E$
and therefore $E_{i+1,i+1} = - E_{i,i}$.
\end{proof}

\begin{lemma}
\label{lemma:hatstep}
Let $\sigma_1 \vartriangleleft \sigma_0 = a_i \sigma_1 = \sigma_1 (j_0 j_1)$, 
$\delta = j_1 - j_0$.
\begin{enumerate}
\item{If $\delta$ is odd then $\hat \sigma_1 \hat a_i   = \hat a_i \hat \sigma_1$
and $\hat \sigma_0 = \hat a_i \hat \sigma_1 = \hat \sigma_1 \hat a_i$.}
\item{If $\delta$ is even then $\hat \sigma_1 \hat a_i = - \hat a_i \hat \sigma_1$
and $\hat \sigma_0 = \hat \sigma_1 = \hat a_i \hat \sigma_1 \hat a_i$.}
\end{enumerate}
\end{lemma}

\begin{proof}
We know (by definition) that
$ \hat\sigma_0 = \acute a_i \hat\sigma_1 \acute a_i $.
As in Lemma \ref{lemma:parityquat}, write
$\hat\sigma_1 = \pm \hat a_1^{\varepsilon_1} \cdots \hat a_n^{\varepsilon_n}$.
We know from Lemma \ref{lemma:pihat}
that we can take $\varepsilon_j = \mult_j(\sigma_{1})$.
Thus
\[ \varepsilon_{i+1} - \varepsilon_{i-1} = (i+1)^{\sigma_{1}} + i^{\sigma_{1}} - (i+1) - i =
j_1 + j_0 -2i-1 \equiv \delta + 1 \pmod 2. \]
If $\delta$ is odd  then
$\hat\sigma_1 \acute a_i = \acute a_i \hat\sigma_1$
and therefore $\hat \sigma_1 \hat a_i = \hat a_i \hat \sigma_1$ and
$\hat \sigma_0 = \hat a_i \hat \sigma_1 = \hat \sigma_1 \hat a_i$. 
If $\delta$ is even then
$\hat\sigma_1 \acute a_i = (\acute a_i)^{-1} \hat\sigma_1$
and therefore $\hat \sigma_1 \hat a_i = - \hat a_i \hat \sigma_1$ and
$\hat \sigma_0 = \hat \sigma_1 = \hat a_i \hat \sigma_1 \hat a_i$. 
\end{proof}

\begin{example}
Using this result it is easy to compute $\hat\sigma_0$ given $\sigma_0$.
Take, say, $\sigma_0 = [7245136] = [a_1a_2a_3a_4a_3a_2a_1a_5a_4a_3a_6]$.
Take
\begin{gather*}
\sigma_0 = a_1 \sigma_1 \vartriangleright \sigma_1 = [2745136], \qquad
\sigma_1 = a_2 \sigma_2 \vartriangleright \sigma_2 = [2475136], \\
\sigma_2 = a_3 \sigma_3 \vartriangleright \sigma_3 = [2457136], \qquad
\sigma_3 = a_4 \sigma_4 \vartriangleright \sigma_4 = [2451736], \\
\sigma_4 = a_3 \sigma_5 \vartriangleright \sigma_5 = [2415736], \qquad
\sigma_5 = a_2 \sigma_6 \vartriangleright \sigma_6 = [2145736], \\
\sigma_6 = a_1 \sigma_7 \vartriangleright \sigma_7 = [1245736], \qquad
\sigma_7 = a_5 \sigma_8 \vartriangleright \sigma_8 = [1245376], \\
\sigma_8 = a_4 \sigma_9 \vartriangleright \sigma_9 = [1243576], \qquad
\sigma_9 = a_3 \sigma_{10} \vartriangleright \sigma_{10} = [1234576] = a_6.
\end{gather*}
We therefore have
\begin{align*}
\hat \sigma_0 &= \hat a_1 \hat\sigma_1 = 
\hat a_1 \hat a_2 \hat\sigma_2 =
\hat a_1 \hat a_2 \hat\sigma_3 =
\hat a_1 \hat a_2 \hat\sigma_4 =
\hat a_1 \hat a_2 \hat\sigma_5 = \\
&= \hat a_1 \hat a_2 \hat a_2 \hat\sigma_6 =
\hat a_1 \hat a_2 \hat a_2 \hat a_1 \hat\sigma_7 =
\hat\sigma_7 = \hat\sigma_8 = \hat\sigma_9 = \hat a_3 \hat\sigma_{10} = \hat a_3\hat a_6, 
\end{align*}
completing the computation. 
\end{example}

\begin{example}
\label{example:hateta}
We have that
$\eta = a_1a_2a_1a_3a_2a_1\cdots a_na_{n-1}\cdots a_2a_1$
is a reduced word so that $ \acute\eta = \acute a_1\acute a_2\acute a_1\acute a_3\acute a_2\acute a_1
\cdots \acute a_n\acute a_{n-1}\cdots \acute a_2\acute a_1$
and $\hat\eta = (\acute\eta)^2$.
From Lemma \ref{lemma:pihat}, we have $\Pi(\hat\eta) = (-1)^n\;I$ and
\[ \Pi(\acute\eta) =
\begin{pmatrix} & & & \iddots \\ & & 1 & \\
 & -1 & & \\ 1 & & & \end{pmatrix}, \qquad \hat\eta = \begin{cases}
1, & n \equiv 0, 6 \pmod 8, \\
-1, & n \equiv 2, 4 \pmod 8, \\ 
\hat a_1 \hat a_3 \cdots \hat a_n, & n \equiv 1, 7 \pmod 8, \\
-\hat a_1 \hat a_3 \cdots \hat a_n, & n \equiv 3, 5 \pmod 8.
\end{cases} \]
Notice the periodicity modulo eight, 
which also occurs in other contexts. 
\end{example}

\begin{rem}
\label{rem:smalln}
For all $n>3$, there are 
$\sigma,\rho\in S_{n+1}\smallsetminus\{e,\eta\}$ 
such that $\hat\sigma=1$ and $\hat\rho=\hat\eta$: 
take $\sigma = (1,5) = a_1a_2a_3a_4a_3a_2a_1$;
we have $\hat\sigma = 1$,
$\longhat(\eta\sigma) = \longhat(\sigma\eta) = \hat\eta$.

No such elements exist in $S_{n+1}$ for $n\in\{2,3\}$.
For $n = 2$ we have $\hat{\sigma}=\pm 1$ if and only if $\sigma\in\{e,\eta\}$, with $\hat{e}=1$ and $\hat\eta = -1$.
For $n = 3$ we have $\hat{\sigma}=\pm 1$ if and only if $\sigma\in\{e,[1432],[3214],[3412]\}$ with $\hat e = 1$,
$\longhat([1432]) = \longhat([3214]) = \longhat([3412]) = -1$. 
Also, $\hat{\sigma}=\pm\hat{\eta}$ if and 
only if $\sigma\in\{[2143],[4123],[2341],\eta\}$, with
$\longhat([2143]) = \longhat([4123]) = 
\longhat([2341]) = \hat a \hat c$,
$\hat\eta = -\hat a\hat c$.
\end{rem}


\section{Triangular coordinates}
\label{sect:triangle}

Let $\Up_{n+1}^{+}\subset\GL_{n+1}$ be the group of real upper triangular matrices with all diagonal entries strictly positive.
Recall the $LU$ decomposition:
a matrix $A \in\GL_{n+1}$ can be (uniquely) written as $A = LU$,
$L \in \Lo_{n+1}^{1}$ and $U \in \Up_{n+1}^{+}$
provided each of its northwest minor determinants is positive.
This condition holds in a contractible open neighborhood 
of the identity matrix $I$;
for $A$ in this set, $L$ and $U$ are smoothly and uniquely defined.
We shall be more interested in $\cU_I \subset \SO_{n+1}$,
the intersection of this neighborhood with $\SO_{n+1}$,
which is also a contractible open subset.
Let $\bL: \cU_I \to \Lo_{n+1}^{1}$ take $Q \in \cU_I$
to the unique $L = \bL(Q) \in \Lo_{n+1}^{1}$ such that
there exists $U \in \Up_{n+1}^{+}$ with $Q = LU$:
the map $\bL$ is a diffeomorphism.
Indeed, its inverse $\bQ: \Lo_{n+1}^{1} \to \cU_I$
is given by the orthogonal factor in the $QR$ decomposition:
given $L \in \Lo_{n+1}^{1}$ let $Q = \bQ(L) \in \SO_{n+1}$
be the unique matrix for which there exists $R \in \Up_{n+1}^{+}$
with $L = QR$.


The set $\Pi^{-1}[\cU_{I}]\subset\Spin_{n+1}$ has two
contractible connected components: we call them $\cU_{1}$ 
and $\cU_{-1}$, where $1\in \cU_{1}$ and $-1\in \cU_{-1}$.
We abuse notation and write $\bL: \cU_1 \to \Lo_{n+1}^{1}$
and $\bQ: \Lo_{n+1}^{1} \to \cU_1$ for the diffeomorphisms
obtained by composition.
For $z_0\in\Spin_{n+1}$, we set $\cU_{z_0}=z_0\cU_1$, 
an open contractible neighborhood of $z_0$, diffeomorphic to 
$\Lo^1_{n+1}$ under the map $z\mapsto\bL(z_0^{-1}z)$. 
This map may be seen as a chart,
defining {\em triangular coordinates}
on the open contractible subset $\cU_{z_0} \subset \Spin_{n+1}$.

For each $j\in\nmesmo$, let $\fl_j= e_{j+1}e_j^\top\in \lo^1_{n+1}$ be the matrix with only one nonzero entry $(\fl_j)_{j+1,j} = 1$.
Recall that $\fa_j=\fl_j-\fl_j^\top\in\so_{n+1}\approx\spin_{n+1}$. 
Let $X_{\fa_j}$ and $X_{\fl_j}$ be the left-invariant vector fields
in $\SO_{n+1}$ and $\Lo_{n+1}^{1}$ generated by $\fa_j$ and $\fl_j$,
respectively:
\[ X_{\fa_j}(Q) = Q \fa_j, \qquad X_{\fl_j}(L) = L \fl_j. \]
We also denote by $X_{\fa_j}$ the corresponding left-invariant vector field
in $\Spin_{n+1}$.

\begin{lemma}
\label{lemma:al}
The diffeomorphisms $\bL: \cU_I \to \Lo_{n+1}^{1}$ and
$\bQ: \Lo_{n+1}^{1} \to \cU_I \subset \SO_{n+1}$ take the vector fields
$X_{\fa_j}$ and $X_{\fl_j}$ to smooth positive multiples of each other.
A similar statement holds for
$\bL: \cU_1 \to \Lo_{n+1}^{1}$ and
$\bQ: \Lo_{n+1}^{1} \to \cU_1 \subset \Spin_{n+1}$.
\end{lemma}

\begin{proof}
Given $Q_{0}\in\cU_{I}$, take a short arc of the 
integral line of $X_{\fa_{j}}$ through $Q_{0}$: 
let $\epsilon>0$ be sufficiently small so that 
$Q(t)=Q_{0}\exp(t\fa_{j})\in\cU_{I}$ for 
$-\epsilon<t<\epsilon$. Also write 
$\bL(Q(t))=L(t)\in\Lo^{1}_{n+1}$, so that 
$L(t)=Q(t)R(t)$ for a smooth path 
$R:(-\epsilon,\epsilon)\to\Up^{+}_{n+1}$. 
Differentiating the last equation, we have
\[(L(t))^{-1}L'(t)=(R(t))^{-1}\fa_{j}R(t)+(R(t))^{-1}R'(t).\]
Since the left hand side is in $\lo^{1}_{n+1}$ 
and the rightmost summand of the right hand side 
is in $\up^{+}_{n+1}$, it is readily seen that 
$L'(t)=(R(t)_{jj}/R(t)_{j+1,j+1})L(t)\fl_{j}$.
\end{proof}


Recall from the introduction that a locally convex curve is 
an absolutely continuous map   
$\Gamma: J \to \Spin_{n+1}$ such that, for all $t \in J$ 
for which the derivative exists, 
the logarithmic derivative 
$(\Gamma(t))^{-1} \Gamma'(t)\in\spin_{n+1}$ is
a positive linear combination of $\fa_1,\ldots,\fa_n$. 
Similarly, a map $\Gamma: J \to \Lo_{n+1}^{1}$ 
is called a \emph{convex curve} if it is absolutely continuous and, 
for all $t \in J$ for which the derivative exists,
the logarithmic derivative $(\Gamma(t))^{-1} \Gamma'(t)$ 
is a positive linear combination of $\fl_1,\ldots,\fl_n$. 

\begin{example}
\label{example:fh}
Consider $\fh_L,\fn\in\lo^1_{n+1}$ and 
$\fh\in\so_{n+1}\approx\spin_{n+1}$ given by 
\[
\fh_L = \sum_{j\in\nmesmo} \sqrt{j(n+1-j)}\; \fl_j, \quad
\fn = \sum_{j\in\nmesmo}  \fl_j, \quad
\fh=\fh_L-\fh_L^{\transpose}\]
We have
\[ [\fh_L,\fh_L^\transpose] = 
\sum^n_{k=0}(2k-n)e_{k+1}e^\top_{k+1}, 
\quad
[\fh_L, [\fh_L,\fh_L^\transpose]] = -2 \fh_L, \quad
[\fh_L^\transpose, [\fh_L,\fh_L^\transpose]] = 2 \fh_L^\transpose,
\]
so that $[\fh_L,\fh_L^\transpose] = \diag(-n,-n+2,\cdots,n-2,n)$.
For $n = 1$ and
$\theta \in (-\frac{\pi}{2}, \frac{\pi}{2})$,
\begin{equation}
\label{equation:fhfhL}
\exp(\theta(\fh_L - \fh_L^\transpose)) =
\exp(\tan(\theta) \fh_L)
\exp(\log(\sec(\theta)) [\fh_L,\fh_L^\transpose])
\exp(-\tan(\theta) \fh_L^\transpose).
\end{equation}
The symmetric product induces a Lie algebra homomorphism
$S: \slalgebra_2 \to \slalgebra_{n+1}$ with $S(\fh_L) = \fh_L$
(that is, taking $\fh_L \in \RR^{2\times 2}$ 
to $\fh_L \in \RR^{(n+1)\times (n+1)}$) and
$S(\fh_L^\transpose) = \fh_L^\transpose$.
We therefore also have a Lie group homomorphism 
$S: \widetilde{\SL_2} \to \widetilde{\SL_{n+1}}$, 
$S(\exp(t\fu)) = \exp(t S(\fu))$ for all $\fu \in \slalgebra_2$, $t\in\RR$.
Equation \ref{equation:fhfhL} therefore holds for any value of $n$.
We therefore have $\exp(\theta\fh) \in \cU_1$ for all 
$\theta \in (-\frac{\pi}{2}, \frac{\pi}{2})$, with
\[ \bL(\exp(\theta \fh)) = \exp(\tan(\theta) \fh_L); \qquad
\bQ(\exp(t \fh_L)) = \exp(\arctan(t) \fh). \]

Also, the equation 
$\exp(\frac{\pi}{2}\fh)=\acute{\eta}$, 
which is trivially true for $n=1$, can be obtained 
for arbitrary $n$ using the Lie group homomorphism $S$  
above and noticing that $S(\acute{\eta})=\acute{\eta}$.  

For $z_0 \in \Spin_{n+1}$, the curve 
$\Gamma_{z_0,\fh}(t) = z_0 \exp(t\fh)$
is locally convex and satisfies 
$\Gamma_{z_0,\fh}(\frac{\pi}{2})=z_{0}\acute{\eta}$, 
$\Gamma_{z_0,\fh}(\pi)=z_{0}\hat{\eta}$. 
For $L_0 \in \Lo^{1}_{n+1}$, the curves
$\Gamma_{L_0, \fh_L}(t) = L_0 \exp(t\fh_L)$ 
and $\Gamma_{L_0, \fn}(t) = L_0 \exp(t\fn)$
are convex.
Notice that the $(i,j)$ entry of
either $\Gamma_{L_0, \fh_L}(t)$ or $\Gamma_{L_0, \fn}(t)$
is a polynomial of degree $(i - j)$ in the variable $t$.
\end{example}

One advantage of working with triangular coordinates
is that there is then a simple integration formula.
Indeed, given a convex curve $\Gamma: J \to \Lo_{n+1}^{1}$, 
write $(\Gamma(t))^{-1} \Gamma'(t) = \sum_i \beta_i(t) \fl_i$. 
The positive functions $\beta_1,\ldots,\beta_n:J\to(0,+\infty)$ 
are then integrable in compact subintervals of $J$. 
Fixed $t_0\in J$, we have 
\[ (\Gamma(t))_{i+1,i} = (\Gamma(t_0))_{i+1,i} +
\int_{t_0}^t \beta_i(\tau) d\tau. \]
More generally,
\begin{equation}
\label{equation:explicitGamma}
((\Gamma(t_0))^{-1} \Gamma(t))_{i+l,i} =
\int_{t_0 \le \tau_1 \le \cdots \le \tau_l \le t}
\beta_{i+l-1}(\tau_1) \cdots \beta_{i}(\tau_{l}) d\tau_1 \cdots d\tau_l. 
\end{equation}
We have, therefore, the following equivalent definition: 
a map $\Gamma: [t_0,t_1] \to \Lo_{n+1}^{1}$
is a convex curve if and only if
there exist finite absolutely continuous (positive) Borel measures
$\mu_1, \ldots, \mu_n$ on $J = [t_0,t_1]$ such that, 
for any index $i\in\nmesmo$ and 
for any nondegenerate interval $\tilde J \subseteq J$, 
$\mu_i(\tilde J)>0$, and such that, for $t_0\leq t\leq t_1$, 
\begin{align}
\label{equation:explicitGammamu}
((\Gamma(t_0))^{-1} \Gamma(t))_{i+l,i}
&= (\mu_{i+l-1} \times \cdots \times \mu_i)(\Delta), \\
\Delta &= \{(\tau_1,\ldots,\tau_l) \in [t_0,t]^k \;|\;
t_0 \le \tau_1 \le \cdots \le \tau_l \le t \}. \nonumber
\end{align}
It follows from Lemma \ref{lemma:al} that a map 
$\Gamma:J\to\Spin_{n+1}$ is locally convex if and only if, 
near any point $t_\bullet \in J$, there is a system of triangular coordinates 
$\Gamma_L(t)=\bL(z_0^{-1}\Gamma(t))$ with 
$\Gamma_L$ convex in the previous sense. 
The reason for calling curves such as $\Gamma_L$ convex 
is that the space curve $\gamma:J\to\RR^{n+1}$ given by 
$\gamma(t)=\Pi(z_0\bQ(\Gamma_L(t)))e_1$ 
is convex in the geometric sense explained in the introduction. 

Let $\nmaisum^{(k)}$ be the set of subsets
$\bi \subseteq \nmaisum$ with $\operatorname{card}(\bi) = k$;
let $\sum(\bi)$ be the sum of the elements of the set $\bi$.
The $k$-th exterior (or alternating) power $\Lambda^k(\RR^{n+1})$
has a basis indexed by $\bi \in \nmaisum^{(k)}$.
For $\bi_0, \bi_1 \in \nmaisum^{(k)}$,
write:
\[ \bi_0 \overset{j}{\to} \bi_1 \quad \iff \quad
j \in \bi_1, \; j+1 \notin \bi_1, \;
\bi_0 =(\bi_1 \smallsetminus \{j\}) \cup \{j+1\}. \]
Notice that $\bi_0 \overset{j}{\to} \bi_1$ implies
$\sum(\bi_0) = 1+ \sum(\bi_1)$.
With respect to the basis above, the matrix of the 
linear endomorphism 
$\Lambda^k(\fl_j)\in\mathfrak{gl}(\Lambda^{k}(\RR^{n+1}))$ given by 
\[\Lambda^k(\fl_j)(v_{1}\wedge\cdots\wedge v_{k})
=\sum_{i\in\llbracket k\rrbracket}v_{1}
\wedge\cdots\wedge\fl_{j}(v_{i})
\wedge\cdots\wedge v_{k}\]
has nonzero entries all equal $1$ and in positions
$(\bi_0, \bi_1)$ such that $\bi_0 \overset{j}{\to} \bi_1$.
Write $\bi_1 \vartriangleleft \bi_0$ if there exists $j$
such that $\bi_0 \overset{j}{\to} \bi_1$ and
define a partial order in $\nmaisum^{(k)}$
by taking the transitive closure.
Equivalently, for
$\bi_j = \{i_{j1} < i_{j2} < \cdots < i_{jk} \}$
we have
\[ \bi_1 \le \bi_0 \quad\iff\quad
i_{11} \le i_{01}, \; i_{12} \le i_{02}, \; \cdots \;, i_{1k} \le i_{0k}. \]
If $\bi_0 \ge \bi_1$, $\sum(\bi_0) = l + \sum(\bi_1)$, write
\[ \bi_0 \overset{(j_1,\ldots,j_l)}{\longrightarrow} \bi_1 \quad\iff\quad
\exists \bj_0,\ldots,\bj_l, \;
\bi_0 = \bj_0 \overset{j_1}{\to} \bj_1 \to \cdots \to
\bj_{l-1} \overset{j_l}{\to} \bj_l = \bi_1; \]
notice that given $\bi_0$ and $\bi_1$
there may exist many such $l$-tuples $(j_1, \ldots, j_l)$.
Order the indices $\bi$ consistently with the partial order
introduced above (or, more directly, order the subsets $\bi$
increasingly in the sum of their elements).
The matrix $\Lambda^k(\fl_i)$ is then strictly lower triangular.

If $L \in \Lo_{n+1}^{1}$ and $\bi_0, \bi_1 \in \nmaisum^{(k)}$,
define $L_{\bi_0,\bi_1}$ to be the $k \times k$ submatrix of $L$
obtained by selecting the rows in $\bi_0$ and the columns in $\bi_1$.
The $(\bi_0,\bi_1)$ entry of $\Lambda^k(L)$ is 
$\det(L_{\bi_0,\bi_1})$.
Clearly, $\bi_0 \not\ge \bi_1$ implies
$\det(L_{\bi_0,\bi_1}) = 0$;
also, $L_{\bi,\bi}$ is lower triangular with diagonal entries
equal to $1$ and therefore $\det(L_{\bi,\bi}) = 1$.
The matrix $\Lambda^k(L)$ is therefore lower triangular
with diagonal entries equal to $1$.
Furthermore, the map $\Lambda^k: \Lo_{n+1}^{1} \to \Lo_{\binom{n+1}{k}}^1$
is a group homomorphism.
The following result generalizes
Equations \ref{equation:explicitGamma} and \ref{equation:explicitGammamu} above.

\begin{lemma}
\label{lemma:explicitGamma}
Let $\Gamma: [t_0,t_1] \to \Lo_{n+1}^{1}$ be a convex curve 
with $\Gamma(t_0) = L_0$ and let
$\beta_i(t) = ((\Gamma(t))^{-1}\Gamma'(t))_{i+1,i}$,
$\mu_i(J) = \int_J \beta_i(t) dt$.
Let $\bi_0, \bi_1 \in \nmaisum^{(k)}$
with $\bi_0 \ge \bi_1$ and $l = \sum (\bi_0) - \sum (\bi_1)$.
Then
\begin{align*}
\det((L_0^{-1} \Gamma(t))_{\bi_0,\bi_1})
&= (\Lambda^k(L_0^{-1} \Gamma(t)))_{\bi_0,\bi_1} 
= \sum_{\bi_0 \overset{(j_1, \ldots, j_l)}{\longrightarrow} \bi_1}
(\mu_{j_1} \times \cdots \times \mu_{j_l})(\Delta); \\
\Delta &= \{(\tau_1,\ldots,\tau_l) \in [t_0,t]^l \;|\;
t_0 \le \tau_1 \le \cdots \le \tau_l \le t \}. 
\end{align*}
\end{lemma}

\begin{proof}
These are straightforward computations.
\end{proof}




\section{Totally positive matrices}
\label{sect:totallypositive}

A matrix $L \in \Lo_{n+1}^{1}$ is \emph{totally positive} if
for all $k \in \nmaisum$ and for all
indices $\bi_0, \bi_1 \in \nmaisum^{(k)}$,
\[ \bi_0 \ge \bi_1 \quad\implies\quad \det(L_{\bi_0,\bi_1}) > 0. \]
Let $\Pos_{\eta} \subset \Lo_{n+1}^{1}$
be the set of totally positive matrices.
In the notation of \cite{Berenstein-Fomin-Zelevinsky}, 
$G=N=\Up^{1}_{n+1}$ and 
$N_{>0}=\Pos_{\eta}^{\transpose}$. 

For each $j\in\nmesmo$, 
let $\jacobi_j(t) = \exp(t \fl_j)$: 
for any reduced word
$\eta = a_{i_1}a_{i_2}\cdots a_{i_m}$, $m = \inv(\eta) = n(n+1)/2$,
the map
\[ (0,+\infty)^m \to \Pos_{\eta}, \qquad
(t_1, t_2, \ldots, t_m) \mapsto
\jacobi_{i_1}(t_1)\jacobi_{i_2}(t_2) \cdots \jacobi_{i_m}(t_m) \]
is a diffeomorphism.
Moreover, there exists a stratification 
of its closure $\overline{\Pos_{\eta}}$:
\[ \overline{\Pos_{\eta}} =
\{ L \in \Lo_{n+1}^{1} \;|\;
\forall \bi_0, \bi_1,
\; ((\bi_0 \ge \bi_1) \to (\det(L_{\bi_0,\bi_1}) \ge 0)) \}
= \bigsqcup_{\sigma \in S_{n+1}} \Pos_{\sigma}; \]
$\Pos_\sigma \subset \Lo_{n+1}^{1}$ is a smooth manifold of dimension $\inv(\sigma)$,
and if $\sigma = a_{i_1}\cdots a_{i_k}$ is a reduced word
(so that $k = \inv(\sigma)$)
then the map
\[(0,+\infty)^k \to \Pos_{\sigma}, \qquad
(t_1, t_2, \ldots, t_k) \mapsto
\jacobi_{i_1}(t_1)\jacobi_{i_2}(t_2) \cdots \jacobi_{i_k}(t_k)
\]
is a diffeomorphism.
Equivalently, if $\sigma_1 \vartriangleleft \sigma_0 = \sigma_1 a_{i_k}$ then the map
\begin{equation}
\label{equation:Possigma}
\Pos_{\sigma_1} \times (0,+\infty) \to \Pos_{\sigma_0}, \qquad
(L,t_k) \mapsto L \jacobi_{i_k}(t_k)
\end{equation}
is a diffeomorphism.

Different reduced words yield different diffeomorphisms
but the same set $\Pos_\sigma$:
the equation
\begin{equation}
\label{equation:ababab}
\jacobi_{i}(t_1) \jacobi_{i+1}(t_2) \jacobi_i(t_3) =
\jacobi_{i+1}\left(\frac{t_2t_3}{t_1+t_3}\right) \jacobi_i(t_1+t_3) 
\jacobi_{i+1}\left(\frac{t_1t_2}{t_1+t_3}\right) 
\end{equation}
provides the transition between adjacent parameterizations
(i.e., between reduced words connected by the local move
in Equation \ref{equation:reducedword2};
the local move in Equation \ref{equation:reducedword1}
corresponds to a mere relabeling).

In general,
the sets $\Pos_{\sigma} \subset \Lo_{n+1}^{1}$ are 
neither subgroups nor semigroups and should not be confused
with the subgroups
$\Lo_\sigma = \Lo_{n+1}^1 \cap (P_\sigma^{-1} \Up_{n+1}^1 P_\sigma)$ of Equation \ref{equation:Upsigma}.
For instance, $\Pos_{e} = \{I\}$ consists of a single point and
$\Pos_{a_i} = \{ \jacobi_i(t), t > 0 \}$ is an open half line;
in this case, $\Pos_{a_i} \subset \Lo_{a_i}$.
For $n = 2$ and
\begin{equation}
\label{equation:Lxyz}
L(x,y,z) =
\begin{pmatrix} 1 & 0 & 0 \\ x & 1 & 0 \\ z & y & 1 \end{pmatrix} 
\end{equation}
we have
\begin{gather*}
\Pos_{ab} = \{ L(x,y,0) \;|\; x, y > 0 \}; \qquad
\Pos_{ba} = \{ L(x,y,xy) \;|\; x, y > 0 \}; \\
\Pos_{aba} = \{ L(x,y,z) \;|\; x, y > 0; \; 0 < z < xy \}.
\end{gather*}
On the other hand,
\[ \Lo_{ab} = \{ L(x,0,z) \;|\; x, z \in \RR \}; \qquad
\Lo_{ba} = \{ L(0,y,z) \;|\; y, z \in \RR \}. \]
If $L \in \Pos_{\sigma}$ then there exist matrices $U_1, U_2 \in \Up_{n+1}$
such that $L = U_1 P_{\sigma} U_2$;
in other words, $\Pos_{\sigma} \subseteq \bQ^{-1}[\Bru_\sigma]$.
The converse is not at all true,
not even if we pay attention to signs of diagonal entries
of the matrices $U_i$.
In \cite{Shapiro-Shapiro-Vainshtein1, Shapiro-Shapiro-Vainshtein2} 
it is shown that the set of matrices
which admit such a decomposition is almost always disconnected;
each cell $\Pos_{\sigma}$ is contractible,
and so is its closure $\overline{\Pos_\sigma}$;
see also Lemma \ref{lemma:posbruhat} below.

\begin{lemma}
\label{lemma:possigma}
Consider $\sigma \in S_{n+1}$, $k \in \nmaisum$ and 
indices $\bi_0, \bi_1 \in \nmaisum^{(k)}$.
If there exists $\sigma_1 = a_{j_1}\cdots a_{j_{l_1}} \le \sigma$,
$\inv(\sigma_1) = l_1$, such that
$\bi_0 \overset{(j_1,\ldots,j_{l_1})}{\longrightarrow} \bi_1$
then, for all $L \in \Pos_\sigma$,
$(\Lambda^k(L))_{\bi_0,\bi_1} > 0$.
Conversely, if no such $\sigma_1$ exists 
then, for all $L \in \Pos_\sigma$,
$(\Lambda^k(L))_{\bi_0,\bi_1} = 0$.
\end{lemma}

\begin{proof}
Write a reduced word $\sigma = a_{i_1}\cdots a_{i_l}$.
Assume first that such $\sigma_1$ exists
and that $j_1 = i_{x_1}, \ldots, j_{l_1} = i_{x_{l_1}}$
(where of course $1 \le x_1 < \cdots < x_{l_1} \le l$).
Set 
\[ \bi_0 = \bj_0 \overset{j_1}{\to} \bj_1 \to \cdots \to
\bj_{l_{1}-1} \overset{j_{l_1}}{\to} \bj_{l_1} = \bi_1; \qquad
L = \jacobi_{i_1}(t_1) \cdots \jacobi_{i_{l}}(t_{l})
\in \Pos_\sigma. \]
We have $ (\Lambda^k(L))_{\bi_0,\bi_1} \ge t_{x_1} \cdots t_{x_{l_1}} > 0 $,
as desired.

Conversely, assume that
$L = \jacobi_{i_1}(t_1) \cdots \jacobi_{i_{l}}(t_{l})$,
$(\Lambda^k(L))_{\bi_0,\bi_1} > 0$.
We have 
\[ (\Lambda^k(L))_{\bi_0,\bi_1} = \sum_{\bi_0 =
\bj_0 \ge \cdots \ge \bj_{l} = \bi_1}
\left( (\Lambda^k(\jacobi_{i_1}(t_1)))_{\bj_0,\bj_1} \cdots
(\Lambda^k(\jacobi_{i_{l}}(t_{l}))
)_{\bj_{l - 1},\bj_{l}} \right).  \]
Consider $(\bj_0, \ldots, \bj_{l})$ such that the above product
is positive.
Let $x_1, \ldots, x_{l_1}$ be such that
$\bj_{x_1-1} > \bj_{x_1},\ldots, \bj_{x_{l_1}-1} > \bj_{x_{l_1}}$:
this obtains a reduced word for $\sigma_1$.
\end{proof}

\begin{lemma}
\label{lemma:positivesemigroup}
Consider $L_0, L_1 \in \Lo_{n+1}^{1}$.
If $L_0 \in \Pos_{\sigma_0}$ and $L_1 \in \Pos_{\sigma_1}$
then $L_0L_1 \in \Pos_{\sigma_0\vee\sigma_1}$.
Thus, $\Pos_{\sigma_0}\Pos_{\sigma_1} = \Pos_{\sigma_0 \vee \sigma_1}$.

In particular, 
if $L_0 \in \Pos_{\eta}$ and $L_1 \in \overline{\Pos_{\eta}}$
then $L_0L_1, L_1L_0 \in \Pos_{\eta}$.
If $L_0 \in \overline{\Pos_{\eta}}$ and $L_1 \in \overline{\Pos_{\eta}}$
then $L_0L_1 \in \overline{\Pos_{\eta}}$.
\end{lemma}

The operation $\vee$ is the one in Example \ref{example:vee}.

\begin{proof}
The first claim can be proved by induction on $l = \inv(\sigma_1)$;
the case $l = 0$ is trivial.
For the case $l = 1$, consider $\sigma_1 = a_i$ and two cases.
If $\sigma_0 \vee a_i = \sigma_0$, we take a reduced word
$\sigma_0 = a_{i_1}\cdots a_{i_k}$ with $i_k = i$. Then
\[ L_0L_1 = (\lambda_{i_1}(t_1) \cdots \lambda_{i_k}(t_k))
\lambda_{i}(t) = \lambda_{i_1}(t_1) \cdots \lambda_{i_k}(t_k + t)
\in \Pos_{\sigma_0}. \]
The case $\sigma_0 \vee a_i \ne \sigma_0$ is even more direct.
The induction step is now easy.

The other claims follow from the first, 
but a direct proof may be instructive:
consider $\bi_0, \bi_1 \in \nmaisum^{(k)}$, $\bi_0 \ge \bi_1$.
If $L_0 \in \Pos_{\eta}$ and $L_1 \in \overline{\Pos_{\eta}}$ we have
\[ (\Lambda^k(L_0L_1))_{\bi_0\bi_1} =
(\Lambda^k(L_0))_{\bi_0\bi_1} (\Lambda^k(L_1))_{\bi_1\bi_1} +
\sum_{\bi_0 \ge \bi > \bi_1}
(\Lambda^k(L_0))_{\bi_0\bi} (\Lambda^k(L_1))_{\bi\bi_1} > 0, \]
as desired; the other cases are similar.
\end{proof}

Write $L_0 \le L_1$ if $L_0^{-1} L_1 \in \overline{\Pos_{\eta}}$
and $L_0 \ll L_1$ if $L_0^{-1} L_1 \in \Pos_{\eta}$;
notice that $L_0^{-1} L_1 \in \Pos_{\eta}$ is in general
not equivalent to $L_1 L_0^{-1} \in \Pos_{\eta}$.
Lemma \ref{lemma:positivesemigroup} implies that these are partial orders:
\begin{equation}
\label{equation:positivesemigroup}
L_0 \le L_1 \le L_2 \;\implies\; L_0 \le L_2; \qquad
L_0 \le L_1 \ll L_2 \;\implies\; L_0 \ll L_2.
\end{equation}

\begin{lemma}
\label{lemma:totallypositive}
Consider $L_0, L_1 \in \Lo_{n+1}^{1}$.
We have that $L_0 \ll L_1$ if and only if
there exists a convex curve $\Gamma: [0,1] \to \Lo_{n+1}^{1}$
with $\Gamma(0) = L_0$ and $\Gamma(1) = L_1$.
\end{lemma}


\begin{proof}
We first prove that the existence of $\Gamma$ implies $L_0 \ll L_1$.
Given $\Gamma$ and $\bi_0, \bi_1 \in \nmaisum^{(k)}$ with $\bi_0 > \bi_1$,
Lemma \ref{lemma:explicitGamma} gives us a formula 
for $(L_0^{-1} L_1)_{\bi_0,\bi_1} > 0$:
$L_0^{-1} L_1$ is therefore totally positive.

Conversely,
let $\fl = \sum_i c_i \fl_i \in \lo_{n+1}^{1}$ for fixed positive $c_i$.
Consider a small closed ball of radius $r > 0$
centered at $L_0^{-1}L_1$ and contained in $\Pos_\eta$,
the image of a continuous map
$h: \BB^m \to \Pos_\eta \subset \Lo_{n+1}^{1}$
with $h(0) = L_0^{-1}L_1$ such that the topological degree of
$h|_{\Ss^{m-1}}$ around $L_0^{-1}L_1$ equals $+1$
(here $m = \dim(\Lo_{n+1}^{1})$).
Consider a fixed reduced word $\eta = a_{i_1} \cdots a_{i_m}$.
Define continuous functions $\tau_i: \BB^m \to (0,+\infty)$
such that $h(s) = \jacobi_{i_1}(\tau_1(s))\cdots \jacobi_{i_m}(\tau_m(s))$.
For $\epsilon \ge 0$, let
\[ \Lambda_\epsilon(s)(t) = m \tau_{j}(s) \fl_{i_j} + \epsilon \fl, \quad
t \in \left(\frac{j-1}{m},\frac{j}{m}\right). \]
Integrate to obtain maps 
\[ \Gamma_\epsilon(s): [0,1] \to \Lo_{n+1}^1, \quad
\Gamma_\epsilon(s)(0) = L_0, \quad
(\Gamma_\epsilon(s)(t))^{-1} (\Gamma_\epsilon(s))'(t) =
\Lambda_\epsilon(s)(t). \]
Notice that $\Gamma_\epsilon(s)$ is a convex curve if $\epsilon > 0$.
Define $h_\epsilon(s) = L_0^{-1} \Gamma_\epsilon(s)(1)$:
clearly $h_0 = h$, i.e., $\Gamma_0(s)(1) = L_0 h(s)$.
By continuity, there exists $\epsilon > 0$ such that for all $s \in \BB^m$
we have $|h_\epsilon(s) - h_0(s)| < r/2$.
The topological degree of $h_\epsilon|_{\Ss^{m-1}}$
around $L_0^{-1}L_1$ equals $+1$.
There exists therefore $s_\epsilon \in \BB^m$
with $h_\epsilon(s_\epsilon) = L_0^{-1}L_1$.
We have that 
$\Gamma = \Gamma_{\epsilon}(s_\epsilon): [0,1] \to \Lo_{n+1}^{1}$
is a convex curve with
$\Gamma(0) = L_0$, $\Gamma(1) = L_1$.
\end{proof}

\begin{rem}
\label{rem:totallypositive}
Minor modifications in the above argument yields a smooth 
convex curve $\Gamma:[0,1]\to\Lo^1_{n+1}$ with $\Gamma(0)=L_0$ 
and $\Gamma(1)=L_1$ if $L_0\ll L_1$. 
\end{rem}

We know by now that if $L_0 \in \Pos_\sigma$ for $\sigma \ne \eta$
and $\Gamma: [0,1] \to \Lo_{n+1}^1$ is a convex curve
with $\Gamma(0) = L_0$ then $\Gamma(t) \in \Pos_\eta$ for all $t > 0$.
The following lemma shows that,
at least from the point of view of certain entries,
the curve $\Gamma$ goes in with positive speed.

\begin{lemma}
\label{lemma:positivespeed}
Given $\sigma \in S_{n+1}$, $\sigma \ne \eta$,
there exist $k \in \nmaisum$ and indices
$\bi_0, \bi_1, \bi_2 \in \nmaisum^{(k)}$ and $j \in \nmesmo$
such that $\bi_0 \ge \bi_1 > \bi_2$,
$\bi_1 \overset{j}{\to} \bi_2$ and, 
for all convex curves $\Gamma: [0,1] \to \Lo_{n+1}^1$
with $\Gamma(0) \in \Pos_\sigma$ and $\Gamma'(0) \ne 0$ (and well defined),
if $g(t) = (\Lambda^k(\Gamma(t)))_{\bi_0,\bi_2}$ then
$g(0) = 0$ and $g'(0) > 0$.
\end{lemma}

\begin{proof}
Consider $k$ and a pair of indices $\bi_0 \ge \bi_3$  
in $\nmaisum^{(k)}$
such that $(\Lambda^k(L))_{\bi_0,\bi_3} = 0$ for $L \in \Pos_\sigma$
(see Lemma \ref{lemma:possigma}).
Keep $k$ and $\bi_0$ fixed and search for $\bi_2 \le \bi_0$ maximal
such that $(\Lambda^k(L))_{\bi_0,\bi_2} = 0$ for $L \in \Pos_\sigma$.
Maximality implies that there exists $\bi_1$,
$\bi_0 \ge \bi_1 > \bi_2$
and an index $j$ such that $\bi_1 \overset{j}{\to} \bi_2$ and
$(\Lambda^k(L))_{\bi_0,\bi_1} > 0$ for $L \in \Pos_\sigma$.

Let $L_0=\Gamma(0)$, $c_0 = (\Lambda^k(L_0))_{\bi_0,\bi_1} > 0$.
Write $h_j(t) = (L_0^{-1} \Gamma(t))_{j+1,j}$
so that $h_j(0) = 0$ and $h'_j(0) = c_j > 0$ 
(see Equation \ref{equation:explicitGamma}).
Now, $g(0)=0$ and, for $t>0$, it follows from $L_0\in\Pos_\sigma$ 
and Lemma \ref{lemma:explicitGamma} that 
\[g(t) \ge
(\Lambda^k(L_0))_{\bi_0,\bi_1} (\Lambda^k(L_0^{-1} \Gamma(t)))_{\bi_1,\bi_2} =
c_0 h_j(t) = c_0 c_j (t+ o(t))\]
(in Landau's small-o notation)
so that $g'(0) \ge c_0c_j > 0$, as desired.
\end{proof}

\begin{rem}
\label{rem:explicitpositivespeed}
We now present an explicit construction. 
Given $\sigma\neq\eta$, take $k$ minimal such that 
$(n-k+2)^\sigma \neq k$. Set then $j=(n-k+2)^\sigma -1$. 
Equivalently, $k$ is minimal such that 
$\Lambda^k(L)_{\bi_0,\bi_3}=0$ for $L\in\Pos_\sigma$, 
$\bi_0=\{n-k+2,\ldots,n+1\}$ and $\bi_3=\{1,\ldots,k\}$.
If we follow the proof of Lemma \ref{lemma:positivespeed}, we have 
$\bi_1=\{1,\ldots,k-1,j+1\}$ and $\bi_2=\{1,\ldots,k-1,j\}$.
\end{rem}

For $\sigma = a_{i_1}\cdots a_{i_k} \in S_{n+1}$ a reduced word, 
and $t_1, \ldots , t_k \in \RR \smallsetminus \{0\}$, let
\begin{equation}
\label{equation:freesign}
L = \jacobi_{i_1}(t_1)\jacobi_{i_2}(t_2) \cdots \jacobi_{i_k}(t_k). 
\end{equation}
It is well known
\cite{Berenstein-Fomin-Zelevinsky,Shapiro-Shapiro-Vainshtein1}
that $L \in \bQ^{-1}[\Bru_\sigma]$.
Let 
\begin{align*}
\Neg_{\sigma} &= X \Pos_{\sigma} X 
= \{L\in\Lo^1_{n+1}\,\vert\, L^{-1}\in \Pos_{\sigma^{-1}}\} \\
&= \{ \jacobi_{i_1}(t_1)\jacobi_{i_2}(t_2) \cdots \jacobi_{i_k}(t_k)\,\vert\,
t_1, t_2, \ldots, t_k \in (-\infty,0) \} 
\end{align*}
where $X = \diag(1,-1,1,-1,\ldots)$ and
$\sigma = a_{i_1}a_{i_2}\cdots a_{i_k}$ is any reduced word
(therefore $k = \inv(\sigma)$).
Of course, each cell $\Neg_\sigma\subset \Lo^{1}_{n+1}$ 
is a contractible submanifold of dimension $\inv(\sigma)$, 
forming the stratification
\[ \overline{\Neg_{\eta}} = \bigsqcup_{\sigma \in S_{n+1}} \Neg_{\sigma}. \]
Notice that
$\overline{\Pos_\eta} \cap \overline{\Neg_\eta} = \{I\}$.

\goodbreak

\begin{lemma}
\label{lemma:transition}
Consider an interval $J\subseteq\RR$ and a convex curve
$\Gamma: J \to \Lo_{n+1}^{1}$.
\begin{enumerate}
\item{If $t_{-1} < t_0 < t_1$ and
$\Gamma(t_0) \in \Pos_{\sigma} \subset \partial\Pos_{\eta}$
for some $\sigma \ne \eta$ then
$\Gamma(t_1) \in \Pos_{\eta}$ and
$\Gamma(t_{-1}) \notin \overline{\Pos_{\eta}}$.}
\item{If $t_{-1} < t_0 < t_1$ and
$\Gamma(t_0) \in \Neg_{\sigma} \subset \partial\Neg_{\eta}$
for some $\sigma \ne \eta$ then
$\Gamma(t_{-1}) \in \Neg_{\eta}$ and
$\Gamma(t_1) \notin \overline{\Neg_{\eta}}$.}
\item{If $t_0 < t < t_1$ then 
$\Gamma(t) \in (\Gamma(t_0) \Pos_{\eta}) \cap (\Gamma(t_1) \Neg_{\eta})$.}
\end{enumerate}
\end{lemma}

\begin{proof}
As in the first item, assume $\Gamma(t_0)\in\Pos_{\sigma}$, 
$\sigma\ne\eta$. From Lemma \ref{lemma:totallypositive}, 
$\Gamma(t_0)\ll \Gamma(t_1)$ and, by definition,  
$\Gamma(t_0)^{-1}\Gamma(t_1)\in\Pos_{\eta}$. 
By Lemma \ref{lemma:positivesemigroup}, 
$\Gamma(t_1)=\Gamma(t_0)\Gamma(t_0)^{-1}\Gamma(t_1)\in\Pos_{\eta}$, 
proving the first claim. 
Assume by contradiction that $\Gamma(t_{-1})\in\overline{\Pos_{\eta}}$: 
from the claim just proved, $\Gamma(t_0)\in\Pos_\eta$, 
a contradiction.
The second item is analogous.
The third item follows from the previous ones.
\end{proof}

\begin{lemma}
\label{lemma:posline}
Consider a reduced word $a_{i_1} \cdots a_{i_m} = \eta$;
consider 
\[ L = \jacobi_{i_1}(t_1) \cdots \jacobi_{i_m}(t_m) \in \Pos_\eta, \qquad
t_1, \ldots, t_m > 0. \]
Then $\jacobi_{i_1}(t) \ll L$ if and only if $t < t_1$ and
$\jacobi_{i_1}(t) \le L$ if and only if $t \le t_1$.
\end{lemma}

\begin{proof}
Let $\sigma_1 = a_{i_1}\eta = a_{i_2} \cdots a_{i_m} \vartriangleleft \eta$; let
\[ L_1 = \jacobi_{i_1}(-t_1) L =
\jacobi_{i_2}(t_2) \cdots \jacobi_{i_m}(t_m) \in \Pos_{\sigma_1} \subset
\overline{\Pos_\eta}. \]
By definition, $\jacobi_{i_1}(t) \ll L$ if and only if
$\jacobi_{i_1}(t_1-t) L_1 \in \Pos_{\eta}$:
this clearly holds for $t < t_1$.
For $t = t_1$, we have 
$\jacobi_{i_1}(t_1-t) L_1 = L_1 \in \Pos_{\sigma_1}$
and therefore $\jacobi_{i_1}(t) \le L$, $\jacobi_{i_1}(t) \not\ll L$.

Finally, assume by contradiction that for $t > t_1$ we have
$\jacobi_{i_1}(t_1-t) L_1 \in \Pos_{\sigma} \subset \overline\Pos_\eta$.
If $a_{i_1}\sigma < \sigma$ consider a reduced word
$\sigma = a_{i_1}a_{j_2}\cdots a_{j_k}$ and write
\[ \jacobi_{i_1}(t_1-t) L_1 =
\jacobi_{i_1}(\tau_1)\jacobi_{j_2}(\tau_2)\cdots\jacobi_{j_k}(\tau_k)
\]
so that
\[ L_1 =
\jacobi_{i_1}(t-t_1+\tau_1)\jacobi_{j_2}(\tau_2)\cdots\jacobi_{j_k}(\tau_k)
\in \Pos_{\sigma},
\]
which implies $\sigma = \sigma_1$, contradicting $a_{i_1}\sigma < \sigma$.
We thus have $a_{i_1}\sigma > \sigma$:
consider a reduced word $\sigma = a_{j_1}\cdots a_{j_k}$  and write
\[ \jacobi_{i_1}(t_1-t) L_1 =
\jacobi_{j_1}(\tau_1)\jacobi_{j_2}(\tau_2)\cdots\jacobi_{j_k}(\tau_k) \]
so that
\[ L_1 =
\jacobi_{i_1}(t-t_1)\jacobi_{j_1}(\tau_1)
\jacobi_{j_2}(\tau_2)\cdots\jacobi_{j_k}(\tau_k)
\in \Pos_{a_{i_1}\sigma},
\]
which implies $a_{i_1}\sigma = \sigma_1$,
contradicting $a_{i_1}\sigma > \sigma$.
\end{proof}

\section{Bruhat cells}
\label{sect:bruhatcell}


In the introduction, we defined the Bruhat stratification of $\Spin_{n+1}$ as the lift of the classical Schubert stratification of the real complete flag variety $\Flag_{n+1}$. 
We now offer an alternative description based on the $UPU$ Bruhat decomposition of invertible matrices:
\[\forall M\in\GL_{n+1}\exists!\,\sigma\in S_{n+1}\exists\, U_0,U_1\in\Up_{n+1}\,\left(M=U_0P_\sigma U_1\right).\]
Notice that the permutation matrix is unique, while the triangular factors are not.
We thus have the partition  
\[\GL_{n+1}=\bigsqcup_{\sigma\in S_{n+1}}\Up_{n+1}P_\sigma\Up_{n+1}\] 
of the real general linear group into double cosets of $\Up_{n+1}$. 
By absorbing signs from $U_0, U_1$ into $P_\sigma$, we may write 
the \emph{signed Bruhat decomposition}: 
\[\forall M\in\GL_{n+1}\exists!\,P\in\B_{n+1}\exists\, \widetilde U_0, \widetilde U_1\in\Up^+_{n+1}\,\left(M=\widetilde U_0P \widetilde U_1\right).\] 
Of course, we have $\sigma_P=\sigma$. 
For each $P\in\B_{n+1}$, the resulting double coset of $\Up^+_{n+1}$ is now a contractible subset of $\GL_{n+1}$, as is its intersection   
with the orthogonal group, which we call a \emph{signed Bruhat cell} 
\cite{Saldanha3, Saldanha-Shapiro}.
In fact, the signed Bruhat cell $\Bru_{P}$ is homeomorphic to the Schubert cell $\mathcal{C}_{\sigma_{P}}\subset\Flag_{n+1}$. 
We have the \emph{signed Bruhat stratification} of the group $\SO_{n+1}$: 
\[\SO_{n+1}=\bigsqcup_{P\in\B^+_{n+1}}\Bru_P,\quad
\Bru_P=\left(\Up^+_{n+1}P\Up^+_{n+1}\right)\cap\,\SO_{n+1}, 
\quad P\in\B^+_{n+1}.\]
The preimage of each cell under the covering map $\Pi: \Spin_{n+1} \to \SO_{n+1}$ is a disjoint union of two contractible components: we call each of these connected components a \emph{signed Bruhat cell} of $\Spin_{n+1}$: for $z \in \widetilde \B^+_{n+1}$, let $\Bru_z$ be the connected component of $\Pi^{-1}[\Bru_{\Pi(z)}]$ containing $z$.
The \emph{unsigned Bruhat cell} $\Bru_\sigma\subset\Spin_{n+1}$, indexed by the permutation $\sigma\in S_{n+1}$, is the disjoint 
union of the signed Bruhat cells $\Bru_z$, $z\in\widetilde\B^+_{n+1}$, such that $\sigma_z=\sigma$. 

Signed Bruhat cells in either $\SO_{n+1}$ or $\Spin_{n+1}$ can also be regarded as the orbits of a certain $\Up^+_{n+1}$-action \cite{Saldanha-Shapiro}. 
For all $U\in\Up^+_{n+1}$ and $Q\in\SO_{n+1}$, 
set $Q^U = \bQ(U^{-1}Q)$. 
This action preserves Bruhat cells and may be lifted to an action
on $\Spin_{n+1}$: we write $z^U = \bQ(U^{-1}z)$. 
Also, if $U \in \Up^+_{n+1}$ and
$\Gamma: [0,1] \to \Spin_{n+1}$ is a locally convex curve, 
then $\Gamma^U: [0,1] \to \Spin_{n+1}$,
$\Gamma^U(t) = \bQ(U^{-1}\Gamma(t))$, is also a locally convex curve.
Also, the nilpotent subgroup $\Up_{n+1}^{1}$ acts simply transitively
on each open Bruhat cell $\Bru_{q\acute\eta}$, $q\in\Quat_{n+1}$, 
and transitively on any Bruhat cell. 
In fact, given $z\in\widetilde\B^+_{n+1}$, the subgroup $\Up_{\sigma_z\eta}$ is the isotropy group of $z$ and the map $U\in\Up_{\sigma_z}\mapsto z^U\in\Bru_z$ is a diffeomorphism (the subgroups $\Up_\sigma\subseteq\Up^1_{n+1}$ were defined in Equation \ref{equation:Upsigma}, Section \ref{sect:symmetric}). 
This already shows that the signed Bruhat cell $\Bru_z$ is a contractible submanifold of dimension $\inv(\sigma_z)$. 

The map $z \mapsto z^U$ 
can be regarded as induced by a projective transformation
\begin{equation}
\label{equation:projtrans} \Ss^n \to \Ss^n, \qquad v \mapsto \frac{U^{-1}v}{|U^{-1}v|};
\end{equation}
we thus say that $\Up_{n+1}^{+}$ acts on $\Spin_{n+1}$
(or $\Bru_\sigma$ or $\Bru_{z_0}$) and on 
locally convex curves by \emph{projective transformations}.


The following result is a simple corollary of these observations;
compare with Lemma \ref{lemma:totallypositive}.

\begin{lemma}
\label{lemma:convex1}
For any $z \in \Bru_{\acute\eta}$ there exists a locally convex curve
$\Gamma: [0,1] \to \Spin_{n+1}$,
$\Gamma(0) = 1$, $\Gamma(\frac12) = z$, $\Gamma(1) = \hat\eta$
and $\Gamma(t) \in \Bru_{\acute\eta}$ for all $t \in (0,1)$.

Moreover, if $h: K \to  \Bru_{\acute\eta}$ is a continuous function
then there exists a continuous function $H: K \times [0,1] \to \Spin_{n+1}$ 
such that for any $s \in K$ the locally convex curve
$\Gamma_s: [0,1] \to \Spin_{n+1}$, $\Gamma_s(t) = H(s,t)$,
satisfies
$\Gamma_s(0) = 1$, $\Gamma_s(\frac12) = h(s)$, $\Gamma_s(1) = \hat\eta$
and $\Gamma_s(t) \in \Bru_{\acute\eta}$ for all $t \in (0,1)$.
\end{lemma}

\begin{proof}
As in Example \ref{example:fh}, take
\[ \fh = \sum_{j\in\nmesmo} \sqrt{j(n+1-j)}\; \fa_j, \qquad
\Gamma_0(t) = \exp(\pi t \fh). \]
Recall that
$\Gamma_0(0) = 1$, $\Gamma_0(\frac12) = \acute\eta$, $\Gamma_0(1) = \hat\eta$.
Equation \ref{equation:fhfhL} implies that, 
for $t\in (0,1)$, $\Gamma_0(t)=\exp\left(\pi\left(t-\frac12\right)\fh\right)=U_{1}(t)\acute{\eta}U_{2}(t)\in\Bru_{\acute{\eta}}$, 
where 
\begin{align*}
U_{1}(t) & =\acute{\eta}\exp\left(-\cot(\pi t)\fh_{L}\right)\acute{\eta}^{-1}\in\Up^{1}_{n+1}, \\
U_{2}(t) & =\exp(-\log(\sin(\pi t))[\fh_L,\fh_L^{\transpose}])
\exp(\cot(\pi t)\fh_L^{\transpose})\in\Up^{+}_{n+1}.
\end{align*}
Define $h_U: K \to  \Up_{n+1}^1$ by $\acute\eta^{h_U(s)} = h(s)$;
define $H(s,t) = (\Gamma_0(t))^{h_U(s)}$ and $\Gamma_s = \Gamma_0^{h_U(s)}$.
\end{proof}

If $q \in \Quat_{n+1}$ then $\Bru_q = \{q\}$.
If $z = q \acute\eta \in \widetilde \B_{n+1}^{+}$, $q \in \Quat_{n+1}$,
then $\Bru_z = \cU_z$, the domain of a triangular system of coordinates 
centered in $z$ (see Section \ref{sect:triangle}). 
If $z = q (\acute a_i)^{\pm 1}$,
$q \in \Quat_{n+1}$,
then $\Bru_z = \{q \alpha_i(\pm\theta)\,\vert\, \theta \in (0,\pi)\}$
where $\alpha_i(\theta) = \exp(\theta \fa_i)$
(recall that $\alpha_i(\pm\frac{\pi}{2}) = (\acute a_i)^{\pm1}$).
Theorem \ref{theo:Bruhat} and 
Corollaries \ref{coro:Bruhat1} and \ref{coro:Bruhat2} generalize 
these observations. 
The diffeomorphism defined by Equation \ref{equation:Possigma} 
is a triangular counterpart to the one in Theorem \ref{theo:Bruhat}. 
A crucial difference between the present case and the triangular case
is that $(0,+\infty)$ and $\Pos_\eta$ are semigroups
(i.e., closed under sums and products, respectively)
but $(0,\pi)$ and $\Bru_{\acute\eta}$ are not.
  
Before presenting a proof of Theorem \ref{theo:Bruhat}, 
we give some applications. 
Notice that Corollaries \ref{coro:Bruhat1} and \ref{coro:Bruhat2} 
follow easily from Theorem \ref{theo:Bruhat}.





\begin{coro}
\label{coro:bruhatproduct}
Consider $\sigma_0, \sigma_1 \in S_{n+1}$, $\sigma = \sigma_0\sigma_1$.
If $\inv(\sigma) = \inv(\sigma_0)+\inv(\sigma_1)$ then
$\Bru_{\acute\sigma_0}\Bru_{\acute\sigma_1} = \Bru_{\acute\sigma}$;
moreover, the map
\[ 
\Bru_{\acute\sigma_0} \times \Bru_{\acute\sigma_1} \to \Bru_{\acute\sigma},
\qquad (z_0,z_1) \mapsto z_0z_1 \]
is a diffeomorphism.
\end{coro}

\begin{proof}
This follows directly from Corollary \ref{coro:Bruhat2}.
\end{proof}




\begin{lemma}
\label{lemma:posbruhat}
Consider $\sigma \in S_{n+1}$.
Then $\bQ[\Pos_{\sigma}] \subset \Bru_{\acute\sigma}$.
Furthermore, if $\sigma \ne e$ then
$\acute\sigma$ does not belong to $\bQ[\Pos_{\sigma}]$.
Similarly, 
$\bQ[\Neg_{\sigma}] \subset \Bru_{\grave{\sigma}}$;
if $\sigma \ne e$ then $\grave{\sigma}$ does not belong to $\bQ[\Neg_{\sigma}]$.
\end{lemma}

\begin{proof}
The case $\sigma = e$ is trivial;
for $\sigma = a_j$ we have
$\Pos_{\sigma} = \{\lambda_j(t)\,\vert\, t > 0\}$
and $\bQ(\lambda_j(t)) = \alpha_j(\arctan(t))$
(where $\alpha_j(\theta) = \exp(\theta \fa_j)$ and
$\lambda_j(t) = \exp(t \fl_j)$). We thus have
\[ \lim_{t \to +\infty} \bQ(\lambda_j(t)) =
\alpha_j\left(\frac{\pi}{2}\right) = \acute a_j, \]
as desired.

We proceed to the induction step.
Assume $\sigma_k = a_{i_1} \cdots a_{i_k}$ (a reduced word)
and $\sigma_{k-1} = a_{i_1} \cdots a_{i_{k-1}} \vartriangleleft
\sigma_k = \sigma_{k-1} a_{i_k}$.
Consider $L_k \in \Pos_{\sigma_k}$; 
write $L_k = L_{k-1} \jacobi_{i_k}(t_k)$,
$t_k \in (0,+\infty)$, $L_{k-1} \in \Pos_{\sigma_{k-1}}$.
By induction, we have $\bQ(L_{k-1}) = z_{k-1} \in \Bru_{\acute\sigma_{k-1}}$.
Consider the curves $\Gamma_L: [0,t_k] \to \Lo_{n+1}^{1}$ and
$\Gamma: [0,t_k] \to \Spin_{n+1}$ defined by
$\Gamma_L(t) = L_{k-1} \jacobi_{i_k}(t)$ and $\Gamma = \bQ \circ \Gamma_L$.
In particular, $\Gamma(0) = z_{k-1}$.
The curve $\Gamma_L$ is tangent to the vector field $X_{\fl_{i_k}}$
and therefore, from Lemma \ref{lemma:al},
the curve $\Gamma$ is tangent to the vector field $X_{\fa_{i_k}}$.
We thus have $\Gamma(t) = z_{k-1} \alpha_{i_k}(\theta(t))$
for some smooth increasing function
$\theta: [0,+\infty) \to [0,+\infty)$.
But $z_{k-1} \in \cU_1$ implies
$z_{k-1} \alpha_{i_k}(\pi) = z_{k-1} \hat a_i \in \cU_{\hat{a}_i}$ 
and therefore $z_{k-1} \alpha_{i_k}(\pi)\notin\cU_1$. 
Thus, we have $\theta: [0,+\infty) \to [0,\pi)$.
From Theorem \ref{theo:Bruhat},
$z_k = \bQ(L_k) \in \Bru_{\acute\sigma_k}$, as desired.

Clearly, for $\sigma \ne e$ we have $\acute\sigma \notin \cU_1$,
implying $\acute\sigma \notin \bQ[\Pos_{\sigma}]$.
The claims concerning $\Neg_{\sigma}$
follow from the claims for $\Pos_{\sigma}$
either by taking inverses or by similar arguments.
\end{proof}

\begin{coro}
\label{coro:zkLk}
Consider
$\sigma_{k-1} \vartriangleleft \sigma_k = \sigma_{k-1} a_{i_k} \in S_{n+1}$.
Consider $z_{k-1} \in \Bru_{\acute\sigma_{k-1}}$
and $z_k \in \Bru_{\acute\sigma_k}$,
$z_k = z_{k-1} \alpha_{i_k}(\theta_k)$, $\theta_k \in (0,\pi)$.
If $z_k \in \bQ[\Pos_{\sigma_k}]$ then $z_{k-1} \in \bQ[\Pos_{\sigma_{k-1}}]$
and $z_{k-1} \alpha_{i_k}(\theta) \in \bQ[\Pos_{\sigma_k}]$
for all $\theta \in (0,\theta_k]$.
\end{coro}

\begin{proof}
Let $\sigma_{k-1} = a_{i_1}\cdots a_{i_{k-1}}$ be a reduced word.  Let 
\[ L_k = \bL(z_k) = \lambda_{i_1}(t_1)\cdots
\lambda_{i_{k-1}}(t_{k-1}) \lambda_{i_k}(t_k). \]
Define $\tilde L_{k-1} = \lambda_{i_1}(t_1)\cdots
\lambda_{i_{k-1}}(t_{k-1})$ and $\tilde z_{k-1} = \bQ(\tilde L_{k-1})$.
The curve $\Gamma_L: [0,t_k] \to \Lo_{n+1}^1$,
$\Gamma_L(t) = \tilde L_{k-1} \lambda_{i_k}(t)$
is taken to $\Gamma = \bQ \circ \Gamma_L$ with
$\Gamma(t_k) = z_k$ and
$\Gamma(t) = \tilde z_{k-1} \alpha_{i_k}(\theta(t))$
for some strictly increasing function $\theta$.
Invertibility of the map $\Phi$ in Theorem \ref{theo:Bruhat}
implies that $\tilde z_{k-1} = z_{k-1}$.
Furthermore, $z_{k-1} \alpha_{i_k}(\theta) = \Gamma(t)$
for some $t \in (0,t_k]$.
\end{proof}

The following result was inspired by conversations
with B. Shapiro and M. Shapiro (see also Section \ref{sect:finalremarkso}).

\begin{coro}
\label{coro:freesign}
Let $\sigma = a_{i_1} \cdots a_{i_k} \in S_{n+1}$ be a reduced word.
Let $t_1, \ldots, t_k \in \RR\smallsetminus\{0\}$;
for $1 \le i \le k$, let $\varepsilon_i = \sign(t_i) \in \{\pm1\}$.
 Let
\[ L = \jacobi_{i_1}(t_1)\jacobi_{i_2}(t_2) \cdots \jacobi_{i_k}(t_k);
\qquad
z =
(\acute a_{i_1})^{\varepsilon_1} \cdots (\acute a_{i_k})^{\varepsilon_k}
\in \widetilde \B_{n+1}^{+}; \]
then $L \in \bQ^{-1}[\Bru_z]$.
\end{coro}

\begin{proof}
The proof is by induction on $k$;
the case $k = 0$ is trivial and the case $k = 1$ is easy.
Let $L_k = L$, $z_k = z$,
\[ L_{k-1} = \jacobi_{i_1}(t_1)\jacobi_{i_2}(t_2) \cdots
\jacobi_{i_{k-1}}(t_{k-1}),
\qquad
z_{k-1} =
(\acute a_{i_1})^{\varepsilon_1} \cdots
(\acute a_{i_{k-1}})^{\varepsilon_{k-1}}
\in \widetilde \B_{n+1}^{+}; \]
by induction hypothesis, 
$\tilde z_{k-1} = \bQ(L_{k-1}) \in \Bru_{z_{k-1}}$.
From Theorem \ref{theo:Bruhat}, we have
$\tilde z_{k-1} \alpha_{i_k}(\theta) \in \Bru_{z_k}$
provided $\sign(\theta) = \varepsilon_k$ and $|\theta| < \pi$;
also, $\tilde z_{k-1} \alpha_{i_k}(\pm\pi) \notin \Bru_\sigma$.
Thus, from Lemma \ref{lemma:al},
$\bQ(L_{k-1} \jacobi_{i_k}(t)) = \tilde z_{k-1} \alpha_{i_k}(\theta(t))$
where $\theta: \RR \to \RR$ is a strictly increasing function
with $\theta(0) = 0$.
As remarked near Equation \ref{equation:freesign},
$L_{k-1} \jacobi_{i_1}(t) \in \bQ^{-1}[\Bru_\sigma]$
for all $t \in \RR \smallsetminus \{0\}$
and therefore $|\theta(t)| < \pi$ for all $t \in \RR$.
Thus, if $\sign(t) = \varepsilon_k$ we have
$\bQ(L_{k-1} \jacobi_{i_k}(t)) \in \Bru_{z_k}$, as desired.
\end{proof}


\begin{proof}[Proof of Theorem \ref{theo:Bruhat}]
Given reduced words 
$\sigma_1=a_{i_1}\cdots a_{i_k}\vartriangleleft 
a_{i_1}\cdots a_{i_k}a_j=\sigma_0$ for 
consecutive permutations in $S_{n+1}$, 
signs $\varepsilon_1,\ldots,\varepsilon_k,\varepsilon
\in\{\pm1\}$, and $q\in\Quat_{n+1}$, 
we want to prove that the map 
$\Phi(z,\theta)=z\alpha_j(\varepsilon\theta)$ is a diffeomorphism 
between $\Bru_{qz_1}\times(0,\pi)$ and $\Bru_{qz_{0}}$, where 
$z_1=(\acute a_{i_1})^{\varepsilon_1}\cdots(\acute a_{i_k})^{\varepsilon_k}$ and 
$z_0=z_1(\acute a_j)^{\varepsilon}$. 
Notice that $\Phi(qz_1,\frac{\pi}{2}) = qz_0$.
We present the case $\varepsilon=+1$; the other case is similar. 

We first prove that for all $z \in \Bru_{qz_1}$ and $\theta \in (0,\pi)$
we have $\Phi(z,\theta) \in \Bru_{qz_0}$. 
By connectivity, 
it suffices to prove that $\Phi(z,\theta) \in \Bru_{\sigma_0}$
(the unsigned Bruhat cell).
Abusing the distinction between $z \in \Spin_{n+1}$ and $\Pi(z) \in \SO_{n+1}$, we write the signed Bruhat decomposition 
$z = U_1 q z_1 U_2 \in \Bru_{qz_1}$. 
Given $\theta \in (0,\pi)$, 
we have $\Phi(z,\theta) = U_1 q z_1 U_2 \alpha_j(\theta)$.
We have $U_2 \alpha_j(\theta) = \acute a_j \lambda_j(t) U_3$
for some $t \in \RR$ and $U_3 \in \Up_{n+1}^{+}$
and therefore
$\Phi(z,\theta) = U_1 q z_0 \lambda_j(t) U_3$.
But since $\sigma_1\vartriangleleft\sigma_1 a_j=\sigma_0$, we have
$z_0 \lambda_j(t) = U_4 z_0$
where $U_4 \in \Up_{n+1}$ has at most a single nonzero nondiagonal entry at position
$(j^{\sigma_1^{-1}}, (j+1)^{\sigma_1^{-1}}) =
((j+1)^{\sigma_0^{-1}}, j^{\sigma_0^{-1}})$.
We have $\Phi(z,\theta) = U_1 q U_4 z_0 U_3\in\Bru_{\sigma_0}$, as desired.

At this point we know that 
$\Phi: \Bru_{z_1} \times (0,\pi) \to \Bru_{z_0}$
is a smooth function.
It is also injective.
Indeed, assume $z \alpha_j(\theta) = \tilde z \alpha_j(\tilde\theta)$.
If $\theta < \tilde\theta$ we have both
$z \in \Bru_{z_1}$ and
$z = \tilde z \alpha_j(\tilde\theta - \theta) \in \Bru_{z_0}$,
contradicting the disjointness of the cells.
The case $\theta > \tilde\theta$ is similar and the case
$\theta = \tilde\theta$ is trivial.

Given $U_2 \in \Up_{n+1}^{+}$, the matrix $\acute a_j U_2$ is almost upper,
with a positive entry in position $(j+1,j)$ 
(recall we are identifying $\acute a_j$ and $\Pi(\acute a_j)$).
There exist unique $r > 0$ and $\theta \in (0,\pi)$ such that
$(\acute a_j U_2)_{j+1,j} = r\sin(\theta)$,
$(\acute a_j U_2)_{j+1,j+1} = r\cos(\theta)$.
The matrix $U_3 = \acute a_j U_2 \alpha_j(-\theta)$ 
also belongs to $\Up_{n+1}^{+}$.
Let $\theta_j: \Up_{n+1}^{+} \to (0,\pi)$, $U_2\mapsto \theta$, 
be the real analytic function defined by the above argument.

Given $z \in \Bru_{qz_0}$, write $z = U_1 qz_0 U_2$,
$U_1,U_2 \in \Up_{n+1}^{+}$.
Notice that 
\[ z \alpha_j(-\theta_j(U_2)) =
U_1 qz_1 (\acute a_j U_2 \alpha_j(-\theta_j(U_2))) =
U_1 qz_1 U_3 \in \Bru_{qz_1}. \]
Thus, $\Phi(z \alpha_j(-\theta_j(U_2)), \theta_j(U_2)) = z$,
proving surjectivity of $\Phi$.
Injectivity implies that even though $U_2$ is not well defined
(as a function of $z$), $\theta_i(U_2)$ is well defined 
(and smooth, again as a function of $z$):
this gives a formula for $\Phi^{-1}$ and proves its smoothness.
\end{proof}


\begin{rem}
\label{rem:bigtheta}
The following real analytic function constructed in the proof above turns out to be useful (see \cite{Goulart-Saldanha}).
Given $q \in \Quat_{n+1}$, $j\in\nmesmo$ and 
$\sigma_0\in S_{n+1}$ such that 
$a_j \le_L \sigma_0$, we define 
$\Theta_j: \Bru_{q\acute\sigma_0} \to (0,\pi)$ as follows: 
write $\sigma_1 \vartriangleleft \sigma_0 = \sigma_1 a_j$
and set $\Theta_j(z) = \theta \in (0,\pi)$
if and only if $z \alpha_j(-\theta) \in \Bru_{q \acute\sigma_1}$.
\end{rem}



Theorem \ref{theo:pathcoordinates} 
gives a transversality condition
between smooth locally convex curves and Bruhat cells. 
More explicitly, given $q\in\Quat_{n+1}$ and $\sigma\in S_{n+1}\smallsetminus\{\eta\}$, let $z_0=q\acute\sigma$. 
We introduce slice coordinates 
$(u_1,\ldots,u_{\inv(\sigma)},f_1,\ldots,f_k)$ 
in an open neighborhood $\cU_{z_0}$ of the non-open 
signed Bruhat cell $\Bru_{z_0}$. 
In these coordinates, 
$\Bru_{z_0}=\{z\in\cU_{z_0}\,\vert\,f_1(z)=\cdots=f_k(z)=0\}$. 
Also, the last coordinate increases along every smooth 
locally convex curve $\Gamma:J\to\Spin_{n+1}$: 
we have $(f_k\circ\Gamma)'(t)>0$ for all $t\in J$. 

\begin{proof}[Proof of Theorem \ref{theo:pathcoordinates}]

We present an explicit construction of the 
coordinate functions $u_i, f_j$. 
Write
$\Lo_{n+1}^{1} = \Lo_{\sigma^{-1}} \Lo_{\sigma^{-1}\eta}$,
i.e., write $L \in \Lo_{n+1}^{1}$ as $L = L_1L_2$,
$L_1 \in \Lo_{\sigma^{-1}}$, $L_2 \in  \Lo_{\sigma^{-1}\eta}$
(see Equation \ref{equation:Upsigma} in Section \ref{sect:symmetric} 
for the subgroups $\Lo_\sigma \subseteq \Lo_{n+1}^1$,
$\Up_\sigma \subseteq \Up_{n+1}^{1}$). 
As in the proof of Theorem \ref{theo:Bruhat}, we ignore the 
distinction between $z \in \Spin_{n+1}$ and $\Pi(z) \in \SO_{n+1}$.
Notice that if $L_1 \in \Lo_{\sigma^{-1}}$ then $z_0 L_1 = U_1 z_0$
for $U_1 = z_0 L_1 z_0^{-1} \in \Up_{\sigma}$.
Thus, every $z \in \cU_{z_0}$ can be uniquely written as
$z = \bQ(U_1 z_0 L_2)$, $U_1 \in \Up_{\sigma}$,
$L_2 \in \Lo_{\sigma^{-1}\eta}$.
Notice that if $U_1, \tilde U_1 \in \Up_{\sigma}$
and  $L_2 \in  \Lo_{\sigma^{-1}\eta}$ then
$\bQ(U_1 z_0 L_2)$ and
$\bQ(\tilde U_1 z_0 L_2)$
belong to the same Bruhat cell.
Also, $z = \bQ(U_1 z_0 L_2) \in \Bru_{z_0}$
if and only if $L_2 = I$.
The maps $u,f$ are defined in terms of
$U_1\in \Up_\sigma$ and 
$z_0 L_2 \in z_0 \Lo_{\sigma^{-1}\eta}$, 
respectively;
in other words, we define affine maps 
$u_U:  \Up_\sigma\to \RR^{\inv(\sigma)}$,  
$f_L:  z_0 \Lo_{\sigma^{-1}\eta} \to \RR^k$ and
set $u(\bQ(U_1 z_0 L_2))=u_U(U_1)$, 
$f(\bQ(U_1 z_0 L_2)) = f_L(z_0 L_2)$. 
From now on, we focus on $f$ ($u$ is similar).

We describe a generic element of the set $z_0 \Lo_{\sigma^{-1}\eta}$. 
Recall we identify $z_0$ with the orthogonal matrix $\Pi(z_0)$.
In order to obtain
$M \in z_0 \Lo_{\sigma^{-1}\eta}$,
we introduce free variables in place of the zeroes of $z_0$
which are below and to the left of nonzero entries.
Call these entries $x_1, \ldots, x_k$,
where we number them in the reading order: top to bottom and left to right.
For each $i\in\nmaisum$, 
apply the sign of the entry $(z_0)_{i,i^\sigma}$ to all of the $i$-th row.
Thus, for instance, an element $z_0$ as below yields a set $z_0\Lo_{\sigma^{-1}\eta}$ with elements $M$ of the following general form:
\[ \Pi(z_0) = \begin{pmatrix}
0 & -1 & 0 & 0 \\ 0 & 0 & 0 & -1 \\ -1 & 0 & 0 & 0 \\ 0 & 0 & 1 & 0
\end{pmatrix} 
\quad \rightarrow \quad
M = \begin{pmatrix}
0 & -1 & 0 & 0 \\ 0 & -x_1 & 0 & -1 \\ -1 & 0 & 0 & 0 \\ x_2 & x_3 & 1 & 0
\end{pmatrix} \in z_0 \Lo_{\sigma^{-1}\eta}. \]
Finally, set $f_L(M) = (x_1, \ldots, x_k)$.
If $x_k$ is in position $(i,j)$ set
$\tilde k = n - i + 2$,
$\bi_0 = \{i, \ldots, n+1\}$ and $\bi_2 = \{1, \ldots, \tilde k - 1, j\}$
(see Lemma \ref{lemma:positivespeed}).
The desired property of
$f_k = \pm (\Lambda^{\tilde k}(M))_{\bi_0,\bi_2}$
follows from
Remark \ref{rem:explicitpositivespeed}.
Equivalently, 
if $\Gamma(t) = z_0 \bQ(\Gamma_L(t))$ then
$f_k(\Gamma(t)) = (\Gamma_L(t))_{j+1,j}$,
which is clearly strictly increasing with positive derivative.
\end{proof}

\begin{rem}
\label{rem:explicitpathcoordinates}
For $z_0\in\widetilde\B^+_{n+1}$, the open set $\cU_{z_0}$ is a 
tubular neighborhood in $\Spin_{n+1}$ 
of the signed Bruhat cell 
$\Bru_{z_0}$, with projection map 
$\Pi_{z_0}: \cU_{z_0} \to \Bru_{z_0}$, 
$\Pi_{z_0}(\bQ(U_1 z_0 L_2)) = \bQ(U_1 z_0)$.  
The smooth map $f=(f_1,\ldots,f_k)$ obtained in Theorem  \ref{theo:pathcoordinates} 
parameterizes transversal sections 
of this tubular neighborhood. 
\end{rem}

We now prove Theorem \ref{theo:chopadvance}.
Consider a locally convex curve
$\Gamma: (-\epsilon,\epsilon) \to \Spin_{n+1}$
with $\Gamma(0) = z$.
We need to prove that there exists $\epsilon_a \in (0,\epsilon]$
such that, for all $t \in (0,\epsilon_a]$ we have
$\Gamma(t) \in \Bru_{\adv(z)}$
(the corresponding claim for $\chop$ is similar).
Recall that
the maps $\chop, \adv: \Spin_{n+1} \to \acute\eta \Quat_{n+1} \subset
\widetilde \B_{n+1}^{+}$
are defined by 
\[
\adv(z) = q_a \acute\eta = z_0 \longacute(\rho_0^{-1}), \quad
\chop(z) \acute\rho_0 = z_0 , \quad
z\in\Bru_{z_0}\subset\Bru_{\sigma_0},  
\]
for
$z_0=q_a \acute\sigma_0$,
$\sigma_0=\sigma_{z_0}$,
$\eta = \sigma_0\rho_0^{-1}$ and $q_a\in\Quat_{n+1}$
(see Equation \ref{equation:chopadvance}).  

\begin{proof}[Proof of Theorem \ref{theo:chopadvance}]
If necessary, apply a projective transformation so that
$z=q_a \bQ(L_0)$, $L_0\in\Pos_{\sigma_0}$.
For any locally convex curve $\Gamma$ as in the statement,
there exists $\epsilon_a\in(0,\epsilon)$ such that
the restriction $\Gamma|_{[-\epsilon_a,\epsilon_a]}$ can be
written in triangular coordinates: 
$\Gamma(t)=q_a \bQ(\Gamma_L(t))$, $\Gamma_L(0)=L_0$.
It follows from Lemma \ref{lemma:transition}
that $\Gamma_L(t) \in \Pos_{\eta}$ for any $t \in (0,\epsilon_a]$.
Thus, $\Gamma(t)\in\Bru_{\adv(z)}$ for all $t\in(0,\epsilon_a]$.
The proof for $\chop$ is similar.
\end{proof}

\section{Multiplicities revisited}
\label{sect:mult}


In this section we present the proof of Theorem \ref{theo:mult}. 
In its statement, the locally convex curves are supposed to be smooth. 
In Lemma \ref{lemma:xmult} below, however,
we consider curves $\Gamma$ of differentiability class $C^r$.
As we shall see, Lemma \ref{lemma:xmult} not only implies
Theorem \ref{theo:mult}
but also the same statement for curves of class $C^r$
with $r \ge r_\bullet=\lfloor\left(\frac{n+1}2\right)^2\rfloor$.


Given a matrix $Q \in \SO_{n+1}$, for each $j\in\nmesmo$ let  
\begin{equation*}
\label{equation:southwestminor}
\swminor(Q,j)=\operatorname{submatrix}
(Q,(n-j+2,\ldots,n+1),(1,\ldots,j))\in\RR^{j\times j}, 
\end{equation*}
be its southwest $j\times j$ block.

Given a locally convex curve $\Gamma:J\to\Spin_{n+1}$, 
for each $j\in\nmesmo$ we define
\begin{equation}
\label{eq:mj}
m_j = m_{\Gamma;j}: J \to \RR,\qquad m_j(t)
=\det(\swminor(\Pi(\Gamma(t)),j)).
\end{equation}
Write $\mult_j(\Gamma;t_0) = \mu$
if $t_0$ is a zero of multiplicity $\mu$
of the function $m_j$, that is, if
$(t-t_0)^{(-\mu)} m_j(t)$
is continuous and non-zero at $t = t_0$.
Notice that for a general locally convex curve $\Gamma$,
$\mult_j(\Gamma;t_0)$ as above is not always well defined.
Let the \emph{multiplicity vector} be
$ \mult(\Gamma;t_0) = \left(
\mult_1(\Gamma;t_0), \mult_2(\Gamma;t_0), \ldots,
\mult_n(\Gamma;t_0)
\right)$
(if each coordinate is well defined). 
Recall that $\Gamma(t_0) \in \Bru_\eta$ if and only if
there exist upper triangular matrices $U_1$ and $U_2$ such that
$\Gamma(t_0) = U_1 \acute\eta U_2$.
It is a basic fact of linear algebra that this happens if and only if
$m_j(t_0) \ne 0$ for all $j$.
Thus, $\Gamma(t_0) \in \Bru_\eta$ if and only if $\mult(\Gamma;t_0) = 0$.

\begin{lemma}
\label{lemma:xmult}
Consider a locally convex curve $\Gamma: J \to \Spin_{n+1}$,
where $J \subseteq \RR$ is an open interval.
Consider $t_0 \in J$ and $\sigma \in S_{n+1}$ such that
$\Gamma(t_0) \in \Bru_\rho$, $\rho = \eta\sigma$.
If $r \ge \mult_j(\sigma)$ for all $j \in \nmesmo$
and $\Gamma$ is of class $C^r$ then
$\mult(\Gamma;t_0)$ is well defined and
$\mult(\Gamma;t_0) = \mult(\sigma)$.
\end{lemma}

Theorem \ref{theo:mult} is a direct consequence of Lemma \ref{lemma:xmult}.
These results can be interpreted as
defining the multiplicity vector for general 
locally convex curves, regardless of their class of differentiability. 
They also justify the notation $\mult(\sigma)$. 
The constant $r_\bullet$ in the first paragraph of this section
is obtained as
$r_\bullet = \mult_j(\eta)$, $j = \lfloor \frac{n+1}{2} \rfloor$,
the smallest value of $r$ for which Lemma \ref{lemma:xmult}
can be applied for any permutation $\sigma \in S_{n+1}$. 
Before we present the proof of Lemma \ref{lemma:xmult},
let us see an easy result in linear algebra.

\begin{lemma}
\label{lemma:vandert}
Let $d_1, d_2, \cdots, d_k$ be non-negative integers.
Let $M$ be the $k\times k$ matrix with entries
\[ M_{i,1} = t^{d_i}, \qquad
M_{i,j+1} = \frac{d}{dt} M_{i,j}. \]
Then
\[ \det(M) = Kt^\mu; \qquad
K = \prod_{i_0 < i_1} (d_{i_1} - d_{i_0}); \qquad 
\mu = -\frac{k(k-1)}{2} + \sum_i d_i. \]
If $\tilde M$ is obtained from $M$
by substituting $1$ for $t$ then $\det(\tilde M) = K \ne 0$.
\end{lemma}

\begin{proof}
We have $M_{i,j} = \tilde M_{i,j} t^{(d_i + 1 - j)}$.
All monomials in the expansion of $\det(M)$ have therefore degree $\mu$.
The first column of $\tilde M$ consists of ones;
the second column has $i$-th entry equal to $d_i$.
The third column has $i$-th entry equal to $d_i(d_i-1) = d_i^2 - d_i$:
an operation on columns leaves the determinant unchanged
but now makes the third column have entries $d_i^2$.
Perform similar operations on columns to obtain a Vandermonde matrix,
implying $\det(\tilde M) = K$, as desired.
\end{proof}

\begin{proof}[Proof of Lemma \ref{lemma:xmult}] 
Assume without loss of generality that $t_0 = 0$ 
and $J = (-\epsilon,\epsilon)$.
Notice that projective transformations 
(defined near Equation \ref{equation:projtrans})
have the effect of multiplying the functions $m_j$
by a positive multiple and therefore do not affect 
the multiplicity vector.
We therefore assume that
$\Gamma(0) = z_0 \in \widetilde\B_{n+1}^{+}$, $\sigma_{z_0}=\rho =\eta\sigma$.
Identifying $z_0$ and the orthogonal matrix $\Pi(z_0)$, as usual, 
we thus have $(z_0)_{i,i^{\eta\sigma}} = \varepsilon_i \in \{ \pm 1\}$
and $(z_0)_{i,j} = 0$ otherwise.
We use generalized triangular coordinates:
$\Gamma_L: (-\epsilon,\epsilon) \to z_0 \Lo_{n+1}^{1}$,
$\Gamma_L(t) = z_0 \bL(z_0^{-1} \Gamma(t))$,
$\Gamma(t) = \bQ(\Gamma_L(t))$.
Notice that $\det(\swminor(\Gamma_L(t),k))$ is a positive multiple
of $\det(\swminor(\Gamma(t),k))$, so that we may work with $\Gamma_L$.
Let $\Lambda_0 = (\Gamma_L(0))^{-1} \Gamma_L'(0) = \sum_i c_i \fl_i$,
$c_i > 0$;
let $C_i = \prod_{j < i} c_j$.

For given $i_0 \in \nmaisum$,
set $j_0 = (n+2-i_0)^{\sigma} = i_0^{\eta\sigma}$.
For $j > j_0$ we have $(\Gamma_L(t))_{i_0,j} = 0$;
also, $(\Gamma_L(t))_{i_0,j_0} = \varepsilon_{i_0} = \pm 1$.
For $j = j_0 - 1$, we have that
the derivative of the function $(\Gamma_L(t))_{i_0,j}$
is a
sufficiently
smooth positive multiple of $(\Gamma_L(t))_{i_0,j+1}$;
we thus have $(\Gamma_L(t))_{i_0,j} = t\;c_j \varepsilon_{i_0} u_{i_0,j}(t)$
where $u_{i_0,j}$ is
sufficiently
smooth and $u_{i_0,j}(0) = 1$.
Similarly, for $j = j_0 - \mu$, $\mu \ge 0$,
we have
$$ 
(\Gamma_L(t))_{i_0,j} = \frac{t^\mu}{\mu!}\;\frac{C_{j_0}}{C_j}\;
\varepsilon_{i_0} u_{i_0,j}(t), \qquad
u_{i_0,j}(0) = 1 $$
or, equivalently,
$$ 
(\Gamma_L(t))_{i,j} =
\frac{1}{(i^{\eta\sigma} - j)!} \;
\frac{\varepsilon_i C_{i^{\eta\sigma}}
t^{i^{\eta\sigma}}}{C_{j} t^{j}} \;
u_{i,j}(t), $$
where we follow the convention that $\frac{1}{\mu!} = 0$ for $\mu < 0$.

Consider now $\det(\swminor(\Gamma_L(t),k))$ as a function of $t$.
Write the entries as above.
The powers of $t$ can be taken out of the determinant,
yielding a factor $t^{\mult_k(\sigma)}$.
The terms $\varepsilon_\ast$ and $C_\ast$ can be taken out,
giving us a nonzero constant multiplicative factor.
Multiply the $i$-th row by $(i^{\eta\sigma} - 1)! \ne 0$:
the remaining matrix $M(t)$ has entries 
\[ M_{i,j}(t) =
\frac{(i^{\eta\sigma}-1)!}{(i^{\eta\sigma} - j)!} \;
u_{i,j}(t). \]
The matrix $\swminor(M(0),k)$ is just like the matrix $\tilde M$
in Lemma \ref{lemma:vandert},
and therefore, $\det(\swminor(M(0),k)) \ne 0$.
By continuity, $\det(\swminor(M(t),k))$ is nonzero near $t = 0$.
\end{proof}

\section{Final Remarks}
\label{sect:finalremarkso}

The content of the present paper was originally conceived
as part of a longer text   proposing a combinatorial approach to the study of the homotopy type of certain spaces of locally convex curves with fixed endpoints 
\cite{Goulart-Saldanha}. 
In a nutshell, 
let $\cL_n$ be the space of
locally convex curves $\Gamma:[0,1]\to\Spin_{n+1}$ 
(say, of class $C^r$)
with $\Gamma(0) = 1$, $\Gamma(1)\in\Quat_{n+1}$.
Theorem \ref{theo:pathcoordinates} implies 
that each $\Gamma \in \cL_n$ intersects
non-open Bruhat cells only for finitely many values 
$0=t_0<t_1<\cdots<t_\ell<t_{\ell+1}=1$
of the parameter $t$. 
We call the finite sequence of permutations 
$\iti(\Gamma)=(\sigma_1,\ldots,\sigma_\ell)
\in(S_{n+1}\smallsetminus\{e\})^\ast$, 
where $\Gamma(t_j)\in\Bru_{\eta\sigma_j}$, 
the \emph{itinerary} of $\Gamma$. 
The space $\cL_n$ is stratified 
into a disjoint union of subspaces of curves with fixed itinerary. 
This stratification, indexed on finite strings of nontrivial permutations, inherits (so to speak) several properties of the Bruhat stratification of $\Spin_{n+1}$, studied in the present paper. 
For instance, Theorem \ref{theo:Bruhat} is used to prove that each strata is contractible; Theorem \ref{theo:pathcoordinates} is a key step in providing each 
strata with the strucutre of a globally collared embedded topological submanifold (the second best thing next to having a smooth tubular neighbohood).
Also, there is a partial order $w_0\preceq w_1$ in the index set $\Word_n=(S_{n+1}\smallsetminus\{e\})^\ast$ that manifests itself as the inclusion between the topological closures of the corresponding strata indexed by the itineraries $w_0,w_1$, in much the same spirit as the Bruhat order. 
It turns out that the differentiability class of the curves 
under consideration plays a significant role in this construction 
\cite{Goulart-Saldanha1}.  
In \cite{Goulart-Saldanha} we use these results to construct
a CW-complex $\cD_n$ homotopically equivalent to $\cL_n$.

We extend the notion of \emph{multiplicity vector}
to $\Word_n$,
setting $\mult(\sigma_1,\ldots,\sigma_\ell)=
\mult(\sigma_1)+\cdots+\mult(\sigma_\ell)$.
One important open question is whether $w_0\preceq w_1$
implies $\mult(w_0)\leq\mult(w_1)$ (as in the Bruhat counterpart).
Without any assumption on the regularity of curves,  
this is essentially equivalent to Conjecture 2.4 in \cite{Shapiro-Shapiro3}. 
Such a result would greatly illuminate the structure of $\cD_n$.
It turns out that working with a space $\cL_n$ of 
sufficiently smooth curves 
allows us to circumvent this difficulty.

Conjecture 2.4 in \cite{Shapiro-Shapiro3}
can be regarded as an attempt at a multiplicative Sturm theory for linear differential ODEs of order $n+1>2$; 
the case $n=1$ corresponding to the classical (additive) Sturm theory. 
The conjecture was proved for $n=2$ in \cite{Shapiro-Shapiro4} and 
recently for $n\le 4$ in \cite{Saldanha-Shapiro-Shapiro},
using some material from the present paper,
particularly Theorem \ref{theo:mult}. 
The said material was also recently applied
(in work in progress with E. Alves, B. Shapiro and M. Shapiro)
to the problem of counting and classifying
connected components of the sets $\bQ^{-1}[\Bru_\sigma] \subseteq \Lo_{n+1}^1$
(for $\sigma \in S_{n+1}$);
Theorem \ref{theo:Bruhat} and Corollary \ref{coro:freesign}
are particularly relevant.

\bibliography{gs}
\bibliographystyle{plain}


\end{document}